\RequirePackage{rotating}
\documentclass[12pt]{amsart}
\usepackage[graphicx]{realboxes}
\usepackage{float}
\restylefloat{table}
\usepackage{placeins}

\usepackage{amsmath}
\usepackage{amssymb}
\usepackage{epsfig}
\usepackage{wasysym}
\usepackage{graphicx}
\usepackage{tikz}
\usepackage{tikz-cd}
\usepackage{bm}
\usepackage{xcolor}
\usepackage{listings}
\numberwithin{equation}{section}
\usepackage{extpfeil}
\usepackage{array}
\usepackage{adjustbox}
\usepackage{longtable}
\usepackage{makecell}
\usepackage[all]{xy}
\usepackage{longtable}
\usepackage{rotating}
\input xy
\xyoption{all}
\usepackage{multirow}
\usepackage{comment}

\usetikzlibrary{positioning}
\usepackage[colorlinks=false,urlbordercolor=white]{hyperref}
  
\tikzset{sgplattice/.style={inner sep=1pt,norm/.style={red!50!blue},char/.style={blue!50!black},
  lin/.style={black!50}},cnj/.style={black!50,yshift=-2.5pt,left=-1pt of #1,scale=0.5,fill=white}}

\DeclareFontFamily{U}{mathb}{\hyphenchar\font45}
\DeclareFontShape{U}{mathb}{m}{n}{
      <5> <6> <7> <8> <9> <10> gen * mathb
      <10.95> mathb10 <12> <14.4> <17.28> <20.74> <24.88> mathb12
      }{}
\DeclareSymbolFont{mathb}{U}{mathb}{m}{n}
\DeclareMathSymbol{\righttoleftarrow}{3}{mathb}{"FD}

\calclayout
\allowdisplaybreaks[3]

\theoremstyle{plain}
\newtheorem{prop}{Proposition}[section]

\newtheorem{theo}[prop]{Theorem}
\newtheorem{coro}[prop]{Corollary}

\newtheorem{lemm}[prop]{Lemma}

\theoremstyle{definition}

\newtheorem{rema}[prop]{Remark}

\def\cC{{\mathcal C}}

\def\cF{{\mathcal F}}
\def\cG{{\mathcal G}}

\def\cI{{\mathcal I}}

\def\cL{{\mathcal L}}
\def\cM{{\mathcal M}}
\def\cN{{\mathcal N}}
\def\cO{{\mathcal O}}
\def\cP{{\mathcal P}}

\def\sA{{\mathsf A}}

\def\fA{{\mathfrak A}}

\def\fD{{\mathfrak D}}

\def\fS{{\mathfrak S}}

\def\mult{{\mathrm{mult}}}

\def\fS{{\mathfrak S}}

\def\bP{{\mathbb P}}
\def\bQ{{\mathbb Q}}

\def\bZ{{\mathbb Z}}

\def\bC{{\mathbb C}}
\def\F{{\mathbb F}}

\newcommand{\mls}{\mathcal M}

\def\bF{{\mathbb F}}

\def\codim{\mathrm{codim}}

\def\Pic{\mathrm{Pic}}

\def\Aut{\mathrm{Aut}}

\def\SL{\mathsf{SL}}

\def\Bir{\mathrm{Bir}}

\def\lim{\mathrm{lim}}

\def\Cr{\mathrm{Cr}}

\def\Sing{\mathrm{Sing}}

\def\Cl{\mathrm{Cl}}

\newcommand{\dq}{\mathbb Q}
\newcommand{\pp}{\mathbb P}
\newcommand{\aq}{{\mathfrak A_4}}
\newcommand{\ac}{{\mathfrak A_5}}
\newcommand{\dc}{\mathbb C}
\newcommand{\dz}{\mathbb Z}
\newcommand{\pl}{\mathbb P^1}
\newcommand{\scinq}{{\mathfrak S_5}}

\newcommand{\st}{{\mathfrak S_3}}
\newcommand{\pic}{\mathrm{Pic}}
\newcommand{\sing}{\mathrm{Sing}}
\newcommand{\p}{\mathbb P}
\newcommand{\rk}{\mathrm{rk}}
\newcommand{\f}{\mathbb F}

\newcommand{\na}{\mathrm{Cr}}

\makeatother
\makeatletter

\begin{document}
\title[quadric threefolds]{$\mathfrak A_5$-equivariant geometry of quadric threefolds}

\author[A. Pinardin]{Antoine Pinardin}
\address{Department of Mathematics, University of Edinburgh, UK}

\email{antoine.pinardin@gmail.com}

\author[Zh. Zhang]{Zhijia Zhang}

\address{
Courant Institute,
  251 Mercer Street,
  New York, NY 10012, USA
}

\email{zz1753@nyu.edu}

\date{\today}

\begin{abstract}
We classify $G$-Mori fibre spaces equivariantly birational to smooth quadric threefolds with fixed-point free actions of the alternating group $G=\mathfrak A_5$. We deduce that such quadric threefolds are $G$-solid and the $G$-actions on them are not projectively linearizable.
\end{abstract}

\maketitle

\section{Introduction}
\label{sect:intro}
One of the central problems in birational geometry is to understand the finite subgroups of the Cremona group $\Cr_n(\bC)$, the group of birational automorphisms of $\bP^n$ over $\bC$. A classification of finite subgroups of $\Cr_2(\bC)$ was obtained in \cite{DI}, but it is far from complete in dimension 3 or higher. A closely related problem is to identify {\em projectively linearizable} subgroups of $\Cr_n(\bC)$, that is, subgroups conjugate to a subgroup induced by a biregular action on $\bP^n$. This has been recently settled in dimension 2 by \cite{PSY}, and is attracting considerable attention in dimension 3, see, e.g. \cite{CTZ,CTZcub,CMTZ,CMTZ3,CTT}. A particularly interesting class of groups is finite simple non-abelian subgroups of $\Cr_3(\bC)$. They have been classified in \cite{ProSimple}. The possibilities are:
$$
\fA_5,\quad \fA_6,\quad \fA_7,\quad \mathsf{PSL}_2(\bF_7),\quad \SL_2(\bF_{8}), \quad\mathsf{PSp}_4(\F_3).
$$

Among them, the most fundamental one in group theory is the alternating group $\fA_5$, the smallest non-abelian simple group. It plays a significant role in birational geometry. There are only three embeddings of $\ac$ in $\na_2(\dc)$, up to conjugation. For their descriptions, see \cite{cheltsov2009two}, \cite{DI} or \cite{bannai2007note}. In contrast, Krylov \cite{krylov2020families} shows that there are infinitely many conjugacy classes of $\fA_5$ in $\Cr_3(\bC)$. Obtaining a classification of all such conjugacy classes is thus a difficult task and remains open. Indeed, there is a wealth of rational threefolds carrying an $\fA_5$-symmetry: the Segre cubic, the Igusa and Burkhardt quartic, the quintic del Pezzo threefold, etc. It is a natural question to ask about conjugation of the corresponding $\fA_5$-actions in $\Cr_3(\bC)$. In the last two decades, this has been extensively studied \cite{Ch-Burk,Avilov-note}, and a book \cite{CS} was written by Cheltsov and Shramov on this topic.  However, the $\ac$-equivariant geometry for one of the simplest Fano threefolds, smooth quadric threefolds, had not been addressed. In this paper, we fill this gap, answering the questions of projective linearizability and solidity.

Throughout, we work over $\bC$ and $G$ is the alternating group $\fA_5$ of order 60, unless otherwise specified. We restrict ourselves to smooth quadric threefolds $X\subset\bP^4$ carrying generically free actions of $G$ such that $X^G\ne\emptyset$, since otherwise a projection from a $G$-fixed point on $X$ yields a $G$-equivariantly birational map $X\dashrightarrow\bP^3$. The arising $G$-actions on $\bP^3$ also have fixed points. %Rationality of quotients of projective spaces by $\fA_5$-actions has been studied by Prokhorov in \cite{prokhorov2025icosahedron}.

Over a non-algebraically closed field, a smooth quadric hypersurface is rational if and only if it has a rational point. Surprisingly, projective linearizability of group actions on quadrics is a more intricate problem, see, e.g., \cite{btquad}. Many obstructions naturally vanish on quadrics, including group cohomology \cite{BogPro,KT-Cremona} and the dual complex of \cite{Esser}, see \cite[Section 2]{CTZcub} for an overview of known obstructions. Non-linearizable actions on quadric threefolds have been found using the Burnside formalism \cite[Example 9.2]{TYZ-3} and birational rigidity \cite{ChS-5,CSZ}. However, the first is not applicable to our case.
%since $\fA_5$ has only small abelian subgroups and centralizers. 

From representation theory, we know that any fixed-point free $G$-action on a smooth quadric threefold $X$ is isomorphic to one of the following two cases, which we refer to as the {\em standard} and {\em nonstandard} actions:
\begin{enumerate}
    \item {\em standard action}: 
    $$
    X=X_1=\left\{x_1^2+x_2^2+x_3^2+x_4^2+x_5^2=0\right\}\subset\bP^4_{x_1,\dots,x_5}
    $$
    with the $G$-action generated by
    \begin{align}\label{eqn:A51act}
        (\mathbf{x})\mapsto (x_2,x_1,x_4,x_3,x_5),\quad
         (\mathbf{x})\mapsto (x_5,x_1,x_2,x_3,x_4).
    \end{align}
    \item {\em nonstandard action}:
     \begin{align}\label{eqn:quadric2}
             X=X_2=\left\{\sum_{1\leq i\leq j\leq 5}x_ix_j=0\right\}\subset\bP^4_{x_1,\dots,x_5}
     \end{align}
    with the $G$-action generated by
    \begin{align}\label{eqn:A52act}
       (\mathbf x)&\mapsto (x_4,x_1,x_5,x_2,-x_1-x_2-x_3-x_4-x_5),  \notag\\
(\mathbf x)&\mapsto (x_4,-x_1-x_2-x_3-x_4-x_5,x_1,x_3,x_2). 
    \end{align}
\end{enumerate}

Our goal is to find all $G$-Mori fibre spaces that are $G$-equivariantly birational to $X_1$ and $X_2$ respectively, using classical techniques from birational rigidity, based on the celebrated Noether--Fano inequalities. The same has been carried out in \cite{CSZ} to show the non-linearizability of the $\fS_5$-action on $X_1$ via the $\fS_5$-permutations on coordinates. Our work generalizes their arguments to other actions.

These quadrics are $G$-equivariantly birational to certain singular cubic threefolds. By \cite[Section 6]{CTZcub}, up to isomorphism, there exists a unique cubic threefold $Y_1$ with $5\sA_1$-singularities and invariant under the $G$-action given by \eqref{eqn:A51act}. 
% It is given by
%\begin{multline}
%Y_1=\{x_1x_2x_3+x_1x_2x_4+x_1x_2x_5+x_1x_3x_4+x_1x_3x_5+x_1x_4x_5+\\+x_2x_3x_4+x_2x_3x_5+x_2x_4x_5+x_3x_4x_5=0\}.
%\end{multline}
  By \cite[Lemma 8.3]{CMTZ}, there is a unique cubic threefold $Y_2$ with $5\sA_2$-singularities and invariant under the $G$-action given by \eqref{eqn:A52act}. See Section~\ref{sect:A5map} for explicit equations of $Y_1$ and $Y_2$. %It is given by
%\begin{multline}\label{eqn:cubic2}
   % Y_2=\{x_1x_2x_3 + x_2x_3x_4 + x_1x_2x_5 + x_1x_4x_5 + x_3x_4x_5\\
%+ \zeta_3(x_1x_2x_4 + x_1x_3x_4 + x_1x_3x_5 + x_2x_3x_5 + x_2x_4x_5)=0\}.
%\end{multline}
 
%It follows that $Y_1$ and $Y_1'$ are $G$-isomorphic. Similarly, $Y_2$ is the unique cubic threefold with $5\sA_2$-singularities and an $\fA_5$-action, see \cite[Section 8]{CMTZ}.

\

Our main results are the following:
\begin{theo}\label{theo:main1}
    The only $G$-Mori fibre spaces that are $G$-equivariantly birational to the quadric threefold $X_1$ are $X_1$ and the cubic threefold $Y_1$.
\end{theo}
\begin{theo}\label{theo:main2}
    The only $G$-Mori fibre spaces that are $G$-equivariantly birational to the quadric threefold $X_2$ are $X_2$ and the cubic threefold $Y_2$.
\end{theo}

There is also a nonstandard $G'=\fS_5$-action on $X_2$, generated by \eqref{eqn:A52act} and the involution
$$
(\mathbf x)\mapsto (x_3,x_4,x_1,x_2,-x_1-x_2-x_3-x_4-x_5).
$$
We prove:
\begin{theo}\label{theo:main3}
    The only $G'$-Mori fibre spaces that are $G'$-equivariantly birational to the quadric threefold $X_2$ are $X_2$ and the cubic threefold $Y_2$.
\end{theo}
A $G$-variety is called $G$-solid if it is not $G$-equivariantly birational to a $G$-Mori fibre space over a positive dimensional base. Our results together with \cite[Theorem 3.1]{CSZ} imply that:

\begin{coro}\label{coro:lin}
  Let $G=\fA_5$ or $\fS_5$, and $X$ a smooth quadric threefold carrying a generically free $G$-action. Then the following are equivalent 
  \begin{itemize}
      \item $G$ does not fix any point on $X$,
      \item the $G$-action on $X$ is not projectively linearizable,
      \item $X$ is $G$-solid.
  \end{itemize}
\end{coro}
Note that all such actions on quadric threefolds are known to be stably linearizable by \cite[Theorem 4.1]{CTZ-uni}.

Here's the roadmap of the paper: in Section~\ref{sect:gen}, we recall basic tools from birational geometry. In Section \ref{sect:A5map}, we present facts about $\fA_5$-equivariant geometry of quadrics. In Sections~\ref{sect:a5quad1} -- \ref{sect:a5cub2}, we prove technical results on singularities of certain log pairs on quadric and cubic threefolds. In Section \ref{sect:theproof}, we prove that these technical results imply Theorem~\ref{theo:main1} and Theorem~\ref{theo:main2}, and 
derive a proof of 
%show that our proof can be adjusted to show
Theorem~\ref{theo:main3}.

\medskip

\noindent
{\bf Acknowledgments:} 
The authors are grateful to Ivan Cheltsov for his careful guidance and detailed feedback on a first draft of the manuscript, to Yuri Tschinkel for his interest and comments, and to Joseph Malbon for helpful discussions. Part of the paper was completed during the semester-long program \textit{Morlet Chair} at CIRM, Luminy. The authors are thankful for its hospitality.

\section{Preliminaries}
\label{sect:gen}

Let $X$ be a projective variety with at most klt singularities and $G$ a finite subgroup of $\Aut(X)$. We use the language of the (equivariant) minimal model program; see, e.g., \cite[I.2]{CheltsovShramov}. Throughout the paper, a log pair $(X,\cM_X)$ refers to a pair consisting of $X$ with a non-empty mobile $G$-invariant linear system $\cM_X$ on $X$ consisting of $\bQ$-Cartier divisors. Let $\pi:\widetilde X\to X$ be a resolution of singularities, and 
$$
\cM^{\widetilde{X}}:=\pi^*(K_X+\cM_X)-K_{\widetilde X}.
$$
For a prime divisor $D\subset\mathrm{Supp}(\cM^{\widetilde{X}})$, the {\em log discrepancy} of the log pair $(X,\cM_X)$ at $D$ is defined as the rational number
$$
a(X,\cM_X;D)=1-\mathrm{mult}_D(\cM^{\widetilde{X}}).
$$

Let $p\in X$ be a point. We say that $(X,\cM_X)$ is {\em canonical (resp. log-canonical, klt)} at $p$ if for any prime divisor $E$ on $\widetilde X$ such that $p\in\pi(E)$, we have $a(X,\cM_X; E)\geq 1$ (resp. $\geq 0$, $>0$).
The {\em non-canonical (resp. non-log-canonical, non-klt) locus} of $(X,\cM_X)$ is the union of points where  $(X,\cM_X)$ is not canonical (resp. not log-canonical, not klt).
%Let $E=\cup_{i\in I} E_i$ be the exceptional divisor of $\pi$ with irreducible components $E_i$. 
%We say the log pair $(X,\cM_X)$ is {\em canonical (resp. log caonical, klt)} if $a(X,\cM_X; E_i)\geq 1$ (resp. $\geq 0$, $>0$) for all $i\in I$.  

An irreducible subvariety $Z\subset X$ is said to be a {\em center of non-canonical (resp. non-log-canonical, non-klt) singularities} of $(X,\cM_X)$ if 
there exists a resolution $\pi:\widetilde X\to X$ and 
 a prime divisor $E$ on $\widetilde{X}$ such that $\pi(E)=Z$ and $a(X,\cM_X;E)<1$ (resp. $<0$, $\leq 0$). For simplicity, we also refer to it as {\em a non-canonical (resp. non-log-canonical, non-klt) center}.

The cornerstone of birational rigidity is the classical Noether--Fano inequality, which reveals a close connection between canonical singularities of log pairs and the existence of birational maps between Mori fibre spaces.  We recall it in the equivariant setting, as is in \cite[Theorem 3.2.6]{CS}.

\begin{theo}\label{theo:nfi}
Let $X$ be a Fano variety with terminal singularities, $G$ a finite subgroup of $\Aut(X)$ such that $\mathrm{rk}(\Cl^G(X))=1$. Assume that there exists a $G$-equivariantly birational map $\chi: X\dashrightarrow V$, where $V$ is a variety with a generically free $G$-action such that one of the following  holds:
\begin{enumerate}
    \item either $V$ is also a Fano variety with terminal singularities such that $\mathrm{rk}(\Cl^G(V))=1$;
    \item  or there exists a $G$-equivariant morphism $V\to Z$ with connected fibres such that its general fibre is a Fano variety, and $Z$ is a normal projective variety with $\dim(V)>\dim(Z)>0$. 
\end{enumerate}
In the former case, let $\cM_X$ be the strict transform on $X$ of the linear system $|-nK_V|$ for $n\gg 0$.
In the latter case, let $\cM_X$ be the strict transform on $X$ of the linear system $|H_V|$, where $H_V$ is the pullback on $V$ of a very ample divisor on $Z$ whose class in $\mathrm{Pic}(Z)$ is $G$-invariant. Let $\lambda\in\dq$ be such that $\lambda\mls_X\sim_\dq-K_X$.
Then, if $\chi$ is not biregular, the log pair $(X,\lambda\mls_X)$ has non-canonical singularities.
\end{theo}

%For two faithful biregular actions of the group $\ac$ on a smooth quadric threefold $X$, we will present a $G$-equivariant birational map from $X$ to a singular cubic $Y$. Then, in a way which will be made precise, we will prove that the non-canonical centers of any log pairs $(X,\lambda\mls_X)$ and $(Y,\mu\mls_Y)$ are small enough to obstruct the existence of any other $G$-Mori fibre space $G$-birational to $X$. We will need the following tools.\\
 The $\alpha$-invariant is a number associated to $X$ which corresponds to the {\em global log-canonical threshold} introduced in \cite{Tian} in a different language. When $X$ is a Fano variety with $\dim(X)\geq 2$, the {\em $G$-equivariant $\alpha$-invariant} of $X$ is the number
$$
\alpha_G(X)=\mathrm{sup}\left\{\lambda\in\mathbb{Q}\ \left|\ \aligned
&\text{the
 pair}\ \left(X, \lambda D\right)\ \text{is log-canonical for any }\\
&\text{$G$-invariant effective $\mathbb{Q}$-divisor}\ D\sim_{\mathbb{Q}} -K_{X}\\
\endaligned\right.\right\}.
$$
We compute this invariant for some $G$-surfaces in Sections~\ref{sect:a5quad1} and~\ref{sect:a5cub1}.  We will use a few results about singularities.

\begin{theo}[{\cite[Theorem 2.1]{Puk4n2}}]\label{theo:4n2}
    Let $X$ be a threefold and $\cM_X$ a non-empty mobile linear system on $X$. If a smooth point $p\in X$ is a non-canonical center of the pair $(X,\lambda\cM_X)$ for some positive rational number $\lambda$, and $D_1, D_2$ are two general elements in $\cM_X$, then
    $$
    \mult_p(D_1\cdot D_2)>\frac{4}{\lambda^2}.
    $$
\end{theo}

Theorem~\ref{theo:4n2} is essentially a corollary of the following theorem due to Corti and the inversion of adjunction; see also \cite[Section 2.5]{CheltsovShramov}.

\begin{theo}[{\cite[Theorem 3.1]{corti1}}]\label{theo:cortiineq}
Let $S$ be a surface and $\cM_S$ a non-empty mobile linear system on $S$. If a smooth point $p\in S$ is a non-log-canonical center of the pair $(S,\lambda\cM_S)$ for some positive rational number $\lambda$, and $D_1, D_2$ are two general elements in $\cM_S$, then
    $$
    \mult_p(D_1\cdot D_2)>\frac{4}{\lambda^2}.
    $$
\end{theo}

In many situations, Theorem~\ref{theo:4n2} gives us a desired bound. However, in certain situations (e.g., in Propositions~\ref{prop:ptcase2} and~\ref{prop:points2A5Y}), a sharper result is needed:

\begin{theo}[{\cite{demailly-pham}}]\label{theo:demailly}
Let $S$ be a smooth surface, $p\in S$ a point, and $\cM_S$ a non-empty mobile linear system on $S$. Assume that $p$ is a non-log-canonical center of the log pair $(S,\lambda\cM_S)$ with some positive rational number $\lambda$. Let
$
m=\mult_p(\cM_S).
$ Then for two general elements $D_1,D_2\in \cM_S$, we have 
    $$
   \mult_p(D_1\cdot D_2)>\frac{m^2}{\lambda^2(m-1)}.
    $$   
\end{theo}
We will also use the following technical observation.

\begin{rema}[{\cite[Remark 3.6]{CSZ}}]
\label{remark:Ziquan}
Let $X$ be a threefold with terminal singularities, $p\in X$ a smooth point,
$\cM_X$ a mobile linear system, and $\lambda\in\bQ_{>0}$. 
If $p$ is a non-canonical center of the log pair $(X,\cM_X)$,
then $p$ is a non-log-canonical center of the log pair
$
(X,\frac32\cM_X)$.
\end{rema}

The Nadel vanishing theorem will give us bounds of the size of $0$-dimensional non-canonical centers.

\begin{theo}[{\cite[Theorem 9.4.8]{lazarsfeld2003positivity}}]\label{theo:nadel}
    Let $X$ be a projective variety with at most klt singularities, $D$ an effective $\dq$-divisor on $X$,  $L$ a Cartier divisor such that $K_X+D+A\sim_\bQ L$ for some ample divisor $A$, and $\mathcal I(X,D)$ the multiplier ideal sheaf of $D$. Then
    $$
        H^i(X,\mathcal O_X(L)\otimes\mathcal I(X,D))=0\text{ for }i\geq 1.
    $$
\end{theo}

Lastly, we introduce some terminology. For a $G$-invariant subvariety $Z\subset X$, we say that $Z$ is {\em $G$-irreducible} if $G$ acts transitively on the irreducible components of $Z$. Let $H$ be a general hyperplane section on $X$. We denote by $|nH-Z|$ the linear system consisting of degree $n$ hyperplane sections on $X$ passing through $Z$. Often, we refer to this as the linear system $|nH-Z|$, although $Z$ is not necessarily a divisor. 

\section{$\ac$-actions on quadric threefolds
}\label{sect:A5map}
Let $X$ be a~smooth quadric threefold carrying a generically free regular action of $G=\fA_5$. Assume that there exists a $G$-orbit $\Sigma$ of 5 points in general position in $X$. Up to a change of variables, we may also assume that the five points are five coordinate points of $\pp^4$. Consider the standard Cremona transformation on $\bP^4$
$$
    \chi\colon(x_1,x_2,x_3,x_4,x_5)\mapsto(\frac{1}{x_1},\frac{1}{x_2},\frac{1}{x_3},\frac{1}{x_4},\frac{1}{x_5}).
$$
The restriction of $\chi$ to $X$ is a $G$-equivariantly birational map. The image $\chi(X)$ is a singular cubic threefold. We say that $\chi$ is the Cremona map associated with $\Sigma$. More descriptions of $\chi$ can be found in \cite{Avilov2016, Avilov2018,CSZ}.\\

Assume that $X^G\ne \emptyset$. From representation theory, there are two possibilities for the $G$-action on the ambient $\bP^4$:
\begin{itemize}
    \item {\em the standard action}: $\bP^4=\bP(\mathbf 1\oplus V_4)$, where $V_4$ is the unique irreducible 4-dimensional representation of $G$,
    \item {\em the nonstandard action}: $\bP^4=\bP(V_5)$, where $V_5$ is the unique irreducible 5-dimensional representation of $G$.
\end{itemize}
\subsection{The standard action}
Under the standard $G$-action on $\bP^4$, up to change of variables, we may assume that $X$ is given by 
$$
\left\{x_1^2+x_2^2+x_3^2+x_4^2+x_5^2=0\right\}\subset\bP^4_{x_1,\dots,x_5}
$$
and the $G$-action is given by $\fA_5$-permutations of 5 coordinates. 
There are two $G$-orbits of length 5, denoted by $\Sigma_5$ and $\Sigma_5'$. Let 
$$
Y_1=\chi_1(X),\quad Y_2=\chi_1'(X)
$$
where $\chi_1$ and $\chi_1'$ are the Cremona maps associated with $\Sigma_5$ and $\Sigma_5'$.
One can check by direct computation that $Y_1$ and $Y_1'$ are cubic threefolds with $5\sA_1$-singularities. By \cite[Section 6]{CTZcub}, such cubics with $\fA_5$-actions are unique up to isomorphism. In particular, we may assume that $Y_1=Y_1'=Y$ where $Y\subset\bP^4$ is given by 
\begin{multline*}
\{x_1x_2x_3+x_1x_2x_4+x_1x_2x_5+x_1x_3x_4+x_1x_3x_5+x_1x_4x_5+\\+x_2x_3x_4+x_2x_3x_5+x_2x_4x_5+x_3x_4x_5=0\}\subset\bP^4_{x_1,\ldots,x_5}
\end{multline*}
 and the $G$-action is still given by permutations of coordinates.

\subsection{The nonstandard action of $\ac$}
Up to isomorphism, we may assume that the $G$-action is as in \eqref{eqn:A52act}. There is a unique $G$-invariant quadric $X\subset\bP^4$, and it is given by the equation \eqref{eqn:quadric2}. There are also two $G$-orbits of length $5$ in $X$. Let $\chi_2$ and $\chi_2'$ be the birational maps associated with them respectively, and 
$$
Y_2=\chi_2(X),\quad  Y_2'=\chi_2'(X
).
$$ 
One can check that $Y_2$ and $Y_2'$ are cubic threefolds with $5\sA_2$-singularities. Such cubic threefolds with $\fA_5$-actions are unique up to isomorphism by \cite[Lemma 8.3]{CMTZ}. Thus, we may assume that $Y=Y_2=Y_2'$ where $Y$ is given by 
$$
Y=\{(8-3\zeta_6)f_1+7f_2=0\}\subset\bP^4,
$$
for

\begin{multline*}
    f_1=x_1^2x_2 + x_1x_2^2 + 2x_1x_2x_3 + x_2^2x_3 + x_2x_3^2 + 2x_2x_3x_4 + x_3^2x_4 + 
        x_3x_4^2 +\\+ x_1^2x_5 + 2x_1x_2x_5 + 2x_1x_4x_5 + 2x_3x_4x_5 + x_4^2x_5 + 
        x_1x_5^2 + x_4x_5^2,
\end{multline*}
\begin{multline*}
    f_2=x_1^2x_3 + x_1x_3^2 + x_1^2x_4 + 2x_1x_2x_4 + x_2^2x_4 + 2x_1x_3x_4 + x_1x_4^2 + 
        x_2x_4^2 + \\+x_2^2x_5 + 2x_1x_3x_5 + 2x_2x_3x_5 + x_3^2x_5 + 2x_2x_4x_5 + 
        x_2x_5^2 + x_3x_5^2,
\end{multline*}
with the same $G$-action given by \eqref{eqn:A52act}.

\section{The standard $\fA_5$-action on the quadric threefold}\label{sect:a5quad1}
Throughout this section, $X$ is the quadric given by 
$$
\left\{x_1^2+x_2^2+x_3^2+x_4^2+x_5^2=0\right\}\subset\bP^4_{x_1,\ldots,x_5}.
$$
Consider the $G$-action on $X$ given by natural $\fA_5$-permutations of the coordinates.   We denote by $\Sigma_5$ and $\Sigma_5'$  two $G$-orbits of length five on $X$. The aim of this section is to prove the following proposition.%Let $\iota\in\mathrm{Aut}(X)$ be the~Galois involution of the double cover $X\to\mathbb{P}^3$ given by the~projection from the~point $(1:1:1:1:1)$. Then $\iota$ swaps $\Sigma_5$ and $\Sigma_5'$ and commutes with the~$G$-action~on~$X$.

\begin{prop}\label{prop:main1}
    Let $\mls_X$ be a non-empty mobile $G$-invariant linear system on $X$, and $\lambda\in\dq$ such that $\lambda\mls_X\sim_\dq-K_X$. Then the log pair $(X,\lambda\mls_X)$ is canonical away from $\Sigma_5\cup\Sigma_5'$.
\end{prop}
\begin{proof}
This follows from Propositions \ref{prop:curvesnotinQ} and \ref{prop:pointsnotinQ}, and Corollary \ref{coro:canonicalinQ}.

\end{proof}

 First, as a guiding principle, we observe that curves of degrees greater than 17 cannot be non-canonical centers of $(X,\lambda\mls_X)$.

\begin{rema}\label{rema:degreebound1}
    If a curve $C$ is a center of non-canonical singularities of $(X,\lambda\mls_X)$, then for two general members $M_1,M_2\in\cM_X$, we have that
$$
\lambda^2(M_1\cdot M_2)=mC+\Delta
$$
for some $m>1$ and some effective divisor $\Delta$ not supported along $C$. Intersecting with a general hyperplane $H$ on $X$, we obtain that 
\begin{align}\label{eqn:bound18}
    18=\lambda^2(M_1\cdot M_2\cdot H)>\deg(C).
\end{align}
\end{rema}

Later, we will see that the size of 0-dimensional non-canonical centers is less than 20, using Nadel vanishing theorem. 

\smallskip

We proceed with subsections. In the first subsection, we classify orbits of length less than 20 and $G$-irreducible curves of degrees at most 17.  In the second subsection, we prove that a $G$-invariant curve not contained in $Q$ cannot be a non-canonical center of $(X,\lambda\mls_X)$, where $Q$ is the unique $G$-invariant hyperplane section on $X$ (cf. Proposition \ref{prop:curvesnotinQ}). In the third subsection, we show that points away from $Q$ and $\Sigma_5\cup\Sigma_5'$ cannot be non-canonical centers (cf. Proposition~\ref{prop:pointsnotinQ}). In the fourth subsection, using the $G$-equivariant $\alpha$-invariant, we prove that no point or curve in $Q$ is a non-canonical center (cf. Corollary \ref{coro:canonicalinQ}).

\subsection{Small $G$-orbits and $G$-invariant curves of low degrees}

\begin{lemm}
\label{lemm:A5orbit}
  A $G$-orbit of points in $X$ with length $<20$ is one of the following:
\begin{align*}
    \Sigma_5&=\text{the orbit of}\quad[1:1:1:2\zeta_4:1],\\
     \Sigma_5'&=\text{the orbit of} \quad[1:1:1:-2\zeta_4:1],\\
       \Sigma_{10}&=\text{the orbit of} \quad[1:1:\frac{\zeta_4\sqrt{6}}{2}:\frac{\zeta_4\sqrt{6}}{2}:1],\\ 
       \Sigma_{10}'&=\text{the orbit of} \quad[1:1:-\frac{\zeta_4\sqrt{6}}{2}:-\frac{\zeta_4\sqrt{6}}{2}:1],\\ 
     \Sigma_{12}&=\text{the orbit of} \quad[1:\zeta_5:\zeta_5^2:\zeta_5^3:\zeta_5^4],\\
        \Sigma_{12}'&=\text{the orbit of}\quad [1:\zeta_5^2:\zeta_5^4:\zeta_5:\zeta_5^3],
\end{align*}
where the length of each orbit is indicated by the subscript.
\end{lemm}
\begin{proof}
    This comes from a computation of fixed points by each subgroup of $G$.
\end{proof}

\begin{lemm}\label{lemm:genericstabcurve}
    Every $G$-invariant curve $C$ in $X$ with $\deg(C)\leq 17$ has a trivial generic stabilizer, that is, the $G$-orbit of a general point in $C$ has length 60. 
\end{lemm}
\begin{proof}
   By computation, we find that all irreducible curves in $X$ with a non-trivial generic stabilizer are conics whose $G$-orbits have length 10 or 15, and thus their degrees exceed 17. 
\end{proof}

There is a distinguished $G$-invariant hyperplane section of $X$ given by
$$
Q=\{x_1+x_2+x_3+x_4+x_5=0\}\cap X.
$$
Note that $Q=\bP^1\times\bP^1$ and that $G$ acts on $Q$ via two non-isomorphic $G$-actions on each copy of $\bP^1$.
Moreover, we have 
$$
\Sigma_5,\Sigma_5^\prime,\Sigma_{10},\Sigma^\prime_{10}\not\in Q,\quad\Sigma_{12}, \Sigma_{12}'\in Q.
$$
Let $B_6$ be the $G$-invariant smooth curve of degree 6 given by
\begin{equation}\label{bring}
\left\{\aligned
&x_1+x_2+x_3+x_4+x_5=0,\\
&x_1^2+x_2^2+x_3^2+x_4^2+x_5^2=0,\\
&x_1^3+x_2^3+x_3^3+x_4^3+x_5^3=0.\\
\endaligned
\right.
\end{equation}
It is known as the Bring curve  \cite[Remark~5.4.2]{CheltsovShramov} and has genus 4.

\begin{lemm}\label{lemm:12line}
    Let $C$ be a $G$-invariant reducible curve in $X$ such that $10<\deg(C)\leq17$. Then $C$ is the union of curves in one of the following $G$-orbits:
    \begin{itemize}
        \item 
        one of the following 2 orbits of 6 conics 
        $$
        \mathcal{C}_{6}=\text{orbit of } C_1,\qquad 
           \mathcal{C}_{6}'=\text{orbit of }C_2, $$
        where \begin{multline*}
        C_1=\{x_1 - x_3 + (-\zeta_{20}^6 + \zeta_{20}^4 + 1)x_4 + (\zeta_{20}^6 - \zeta_{20}^4 - 1)x_5=\\=
x_2 + (\zeta_{20}^6 - \zeta_{20}^4 - 1)x_3 + (-\zeta_{20}^6 + \zeta_{20}^4 + 1)x_4 - x_5=0\}\cap X,
    \end{multline*}
     \begin{multline*}
              C_2=\{x_1 - x_3 + (\zeta_{20}^6 - \zeta_{20}^4)x_4 + (-\zeta_{20}^6 + \zeta_{20}^4)x_5=\\=
x_2 + (-\zeta_{20}^6 + \zeta_{20}^4)x_3 + (\zeta_{20}^6 - \zeta_{20}^4)x_4 - x_5=0\}\cap X.
    \end{multline*}
       
        \item 
    one of the following 2 orbits of 12 lines
     \begin{multline*}
       \mathcal L_{12}= \text{the orbit of the line } \{x_1+\zeta_5x_4+(\zeta_5^3+\zeta_5+1)x_5=x_2+\\+(\zeta_5^3+1)x_4+(\zeta_5^2+1)x_5=x_3-(\zeta_5^3+\zeta_5)x_4+\zeta_5^4x_5=0\},
    \end{multline*}
     \begin{multline*}
               \mathcal L'_{12}= \text{the orbit of the line } \{x_1 + \zeta_5^2x_4 + (\zeta_5^2 + \zeta_5 + 1)x_5=
    x_2 +\\+ (\zeta_5 + 1)x_4 - (\zeta_5^3 + \zeta_5^2 + \zeta_5)x_5=
    x_3 - (\zeta_5^2 + \zeta_5)x_4 + \zeta_5^3x_5=0\}.
    \end{multline*}
    \end{itemize}
    Each of the orbits above consists of pairwise disjoint components. The orbits $\cL_{12}$ and $\cL_{12}'$ are contained in $Q$. The orbits $\cC_6$ and $\cC_6'$ are not.
\end{lemm}

\begin{proof}

From indices of strict subgroups of $G$, we find that $\deg(C)=12$ or 15. If $\deg(C)=15$, then $C$ is a union of 5 twisted cubics. Each of the twisted cubic receives a generically free $\fA_4$-action and spans a $\bP^3$. The $\fA_4$-action on $\bP^3$ should have two invariant lines. We check that this does not happen for the given $\fA_4$-action in our case. So this case is impossible.

If $\deg(C)=12$, then $C$ is either a union of 6 conics or 12 lines. If $C$ contains a conic, the plane spanned by the conic is left invariant by a subgroup $\fD_5\subset G$. We check that the unique (up to conjugation) $\fD_5$ in $G$ leaves invariant two planes in $\bP^4$, giving rise to $
\cC_6$ and $\cC_6'$. If $C$ consists of 12 lines,  each line is left invariant by some subgroup $C_5\subset G$, and thus contains two $C_5$-fixed points. Then, a computation of $C_5$-fixed points leads us to $\cL_{12}$ and $\cL_{12}'$.
\end{proof}

\begin{lemm}
\label{lemma:quadric-curves in S2}
Let $C$ be a~$G$-invariant curve in $X$ with $\mathrm{deg}(C)\leqslant 10$. Then $C$ is contained in $Q=\bP^1\times\bP^1$, and is one of the following:
\begin{itemize}
\item a smooth irreducible curve of bidegree $(1,7)$  and genus 0,
     \item a smooth irreducible curve of bidegree $(2,6)$ and genus 5,
       \item the Bring curve $B_6$ of bidegree $(3,3)$ and genus 4,
      \item a smooth irreducible curve of bidegree $(4,4)$ of genus 9,
       \item a union of 5 conics of bidegree $(5,5)$.
\end{itemize}
\end{lemm}
\begin{proof}
    Assume that $C$ is not contained in $Q.$ Then $Q\cdot C=\mathrm{deg}(C)$ and $Q\cap C$ consists of a $G$-orbit of points in $Q$ of length $\deg(C)$. From the information of orbits in Lemma \ref{lemm:A5orbit}, we see that $\deg(C)\geq12$. Thus, the curve $C$ is contained in $Q.$ A computation of $G$-invariant divisors in $Q$ of bidegree $(r_1,r_2)$ with $r_1+r_2\leq10$ completes the proof.
\end{proof}

Now, we want to classify the $G$-invariant irreducible curves of degrees at most 17 which are not contained in $Q$. The strategy is that for each such curve $C$, we find a $G$-invariant K3 surface containing $C$ and use the geometry of the K3 surface to proceed. In particular, we are interested in the pencil $\cP$ consisting of $G$-invariant K3 surfaces on $X$ given by
    $$
    S_{a_1,a_2}:=\{a_1 f^3+a_2 g=0\}\cap X,\quad [a_1:a_2]\in\bP^1
    $$
where
$$
f=\sum_{i=1}^5x_i\quad \text{and}\quad g=\sum_{i=1}^5x_i^3.
$$

Note that the base locus of $\cP$ is the Bring curve $B_6$. We can find singular members in $\cP$ by direct computations.
\begin{lemm}\label{lemm: sing S}
    A surface $S_{a_1,a_2}$ in $\cP$ is reduced and singular if and only if one of the following holds:
    \begin{itemize}
        \item $[a_1:a_2]=[4\pm 3\zeta_4: 50]\in\bP^1$. In these cases, $\Sing(S_{a_1,a_2})$ consists of $5$ nodes.
        \item $[a_1:a_2]=[6\pm \zeta_4\sqrt{3/2}:75]\in\bP^1$. In these cases, $\Sing(S_{a_1,a_2})$ consists of $10$ nodes.
       % \item $\mu=0$ and $S=Q$. 
    \end{itemize}
    Moreover, when $S_{a_1,a_2}$ is smooth, it does not contain $\Sigma_5,\Sigma_5',\Sigma_{10}$ or $\Sigma_{10}'.$
\end{lemm}

\begin{rema}\label{rema:6conicS}
    The orbits $\cC_6$ and $\cC_6'$ are contained in $S_{2,25}\in\cP$.
\end{rema}

\begin{lemm}
\label{lemma:curves not in Q}
Let $C$ be an $G$-invariant curve not contained in $Q$ such that $\mathrm{deg}(C)\leqslant 17$. Then the following statements hold.
\begin{enumerate}
    \item $\mathrm{deg}(C)=12$.
    \item In the pencil $\cP$, there is a unique surface $S$ containing $C$.
    \item If $C$ is irreducible, then $C$ is a Cartier divisor on $S$.
    \item The surface $S$ is smooth.
    \item The curve $C$ is smooth.
    \item If $C$ is irreducible, then its genus $g(C)\in\{0,5,10\}$.
    \item There exists a $G$-invariant curve $C'$ different from $C$ such that $C'\subset S$, $C'$ is isomorphic to $C$, and $C+C'\sim_{\mathbb Q}\mathcal O_S(4)$.
\end{enumerate}
\end{lemm}
\begin{proof}We may assume that $C$ is $G$-irreducible.
    \begin{enumerate}
        \item Arguing as in Lemma~\ref{lemma:quadric-curves in S2}, we know that $\deg(C)=12.$
    \item Let $P$ be a general point on $C$.  There exists a unique $S\in\cP$ such that $P\in S$. If the curve $C$ is not contained in $S$, then the number of points in $C\cap S$ is at most $3\deg(C)=36$. But by Lemma \ref{lemm:genericstabcurve}, the $G$-orbit of $P$ has length $60$. By contradiction, we see that $C\subset S$. 
        \item In what follows, we will denote by $H$ a general hyperplane section on $X$, and by $H_S$ its restriction to $S$. If the curve $C$ is contained in the smooth locus of $S$, then it is Cartier. Assume that $C\cap\Sing(S)$ is not empty. Let $f\colon \widetilde S\rightarrow S$ be the blowup of $C\cap\Sing(S)$, and $\widetilde C$ the strict transform of $C$ by $f$. We have 
        $$
        \widetilde C\sim_\dq f^*(C)-mE,\quad m\in\frac{1}{2}\dz,
        $$ 
        where $E$ is the exceptional divisor of $f$. To show that $C$ is Cartier, it suffices to prove that $m$ is an integer. Denoting by $E_P$ the component of $E$ mapped to $P$, we have $\widetilde C\cdot E_P=2m$. But this intersection number is preserved by the action of the stabilizer of $P$. By Lemma \ref{lemm: sing S}, we have $s=|C\cap\Sing(S)|\in\{5,10\}$. If $s=5$, the stabilizer of $P$ is $\aq$, and $2m=4a+6b+12c$, where $a,b,c\in\bZ_{\geq 0}$, since $4,6,$ and $12$ are the possible lengths of $\aq$-orbits on $\pl$. It follows that $m$ is an integer, and $C$ is Cartier. If $s=10,$ then the stabilizer of $P$ is $\st$, and $2m=2a+3b+6c$. If $b=0$, then $m$ is an integer and we are done. Assume that $b\ge1$. We have 
        $$
        \widetilde C^2=(f^*(C)-mE)^2=C^2-2sm^2\le C^2-35.
        $$
        By Hodge index theorem, we have 
        $$
        C^2\leq \frac{(C\cdot H_S)^2}{(H_S)^2}=24,\quad\text{and}\quad \widetilde C^2\leq-11,
        $$
        which is impossible and we obtain a contradiction.
        \item  If $C$ is reducible, the assertion follows from Remark~\ref{rema:6conicS}. Assume that $C$ is irreducible. From \cite[Proposition 6.7.3]{CS}, we know that $\mathrm{rk}(\Pic^G(S))=$1 or 2, and $S$ is smooth in the latter case. Assume that $S$ is singular, then $\Pic^G(S)=\bZ$ and it is generated by $H_S$ since $(H_S)^2=6$ is not a square. It follows that $C\sim nH_S$, for some positive integer $n$.  Note that $\deg(C)=12$ implies that $n=2$. But one can check that all $G$-invariant quadratic forms on $\bP^4$ are linear combinations of $\sum_{i=1}^5x_i^2$ and $(\sum_{i=1}^5x_i)^2$. We deduce that no $G$-invariant curve in $S$ is linearly equivalent to $2H_S$, hence we obtain a contradiction.
        \item If $C$ is reducible, the assertion follows from Lemma \ref{lemm:12line}.  Assume that $C$ is irreducible and singular, the singular locus of $C$ is a union of $G$-orbits. Since $S$ is smooth, the curve $C$ does not contain any orbit of length $\leq$10, by Lemma \ref{lemm:A5orbit}. Hence, $\mathrm{Sing}(C)$  must be an orbit of length at least $12$. Let us show that this is impossible. Again, Hodge index theorem gives $$
        C^2\leq \frac{(C\cdot H_S)^2}{(H_S)^2}=24.
        $$ 
        If this is an equality, then $C\sim nH_S$, for some $n\in\dz$, and we have proved that this is impossible. So we have $C^2<24$, and since the self-intersection of a curve on a K3 surface is even, we get $C^2\le22$. It follows that the arithmetic genus $p_a(C)$ of $C$ satisfies $C^2=2p_a(C)-2$, i.e., $p_a(C)\le12$. Thus, $C$ cannot have more than 12 singular points. If $C$ has 12 singular points, since all orbits of length 12 are in $Q$, we have $12=Q\cdot C \geq 2\cdot 12=24$, which is a contradiction.
        \item We have proved that $p_a(C)\le12$ and that $C$ is smooth, so its genus $g(C)\le12$. %By \cite[Lemma 5.1.4]{CS}, the quotient of $C$ by $\ac$ is of genus $0$. Notice that $C$ is not elliptic, because $\ac$ does not act on elliptic curves. Using the database of smooth irreducible curves with an action of $\ac$ in \cite{lmfdb} such that the quotient is $\mathbb P^1$, 
        Note that $C$ only contains one orbit of length 12 since $C\cdot Q=12$. The using a classification of genera of smooth irreducible curves with  $\fA_5$-actions and their orbit structures \cite[Lemma 5.1.5]{CS}, we deduce that $g(C)\in\{0,5,10\}$.
        \item Consider the action of $\mathfrak S_5$ given by the permutations of the coordinates leaving $X$ and $S$ invariant. By \cite{CSZ}, there is no $\scinq$-invariant irreducible curve of degree 12 not contained in $Q$. Let $C'$ be the other curve in the $\scinq$-orbit of $C$. Since $C+C'$ is of degree 24 and since $\pic^\scinq(S)=\dz\cdot H_S$, we get $C+C'\sim4H_S$.
    \end{enumerate}
\end{proof}

\subsection{Invariant curves not contained in $Q$}
This subsection is devoted to proving the following. 
\begin{prop}\label{prop:curvesnotinQ}
    If $C$ is a $G$-invariant curve in $X$ not contained in $Q$, then each irreducible component of $C$ is not a non-canonical center of $(X,\lambda\cM_X)$.
\end{prop}
\begin{proof}
This follows from Lemmas \ref{lemma:irrcurves not in Q} and \ref{lemma:redcurves not in Q}.
\end{proof}

We start with the case of irreducible curves. The method of the proof will be applied several times in this paper. 
\begin{lemm}\label{lemma:irrcurves not in Q}
    If $C$ is an irreducible $G$-invariant curve not contained in $Q$, then $C$ is not a non-canonical center of $(X,\lambda\cM_X)$. 
\end{lemm}
\begin{proof}
By Lemma~\ref{lemma:curves not in Q}, the curve $C$ is of degree 12, and there exists a unique smooth K3 surface $S$ in the pencil $\cP$ such that $C\subset S$, and the genus $g=g(C)\in\{0,5,10\}$. Let $H$ be a general hyperplane section on $X$, and $H_S$ its restriction to $S$. Assume that $C$ is a non-canonical center of $(X,\lambda\mls_X)$. Then $\mult_C(\lambda\mathcal M_X)>1$. We have 
$$
\lambda\mls_X\vert_S\sim_{\dq}mC+\Delta,\quad m\ge\mult_C(\lambda\mathcal M_X)>1
$$
for some divisor $\Delta$ on $S$ not supported along $C$. In particular, the divisors 
$$
3H_S-C\sim_\dq\Delta+(m-1)C\quad\text{and}\quad 3H_S-mC\sim_\bQ\Delta
$$ are effective. By Lemma \ref{lemma:curves not in Q}, there exists an irreducible curve $C'$ such that $C'$ is isomorphic to $C$ and $C'\sim_\dq 4H_S-C$.  
    \begin{enumerate}
        \item Assume that $g=0$. We have $(C')^2=(4H_S-C)^2=-2$. So the divisor $4H_S-C$ is on an extremal ray of the Mori cone of $S$. Since $H_S$ is ample, it implies that $C'-H_S\sim_\mathbb{Q} 3H_S-C$ is not rationally equivalent to any effective divisor. Hence, we get a contradiction.
        \item Assume that $g=5$. Notice that $C'$ is nef since it is an irreducible curve on a smooth surface, and $(C')^2=2g(C')-2=8$. But $$
        (3H_S-C)\cdot C'=(3H_S-C)\cdot(4H_S-C)=-4<0,
        $$
        which gives a contradiction.
        \item Assume that $g=10$. Let us first show that the linear system $|3H_S-C|$ has no fixed part. Notice that its mobile part is at least a pencil. Indeed, by Riemann-Roch theorem, we have $$h^0(3H_S-C)\ge2+\frac{1}{2}(3H_S-C)^2=2.$$ So, if it has a base curve, it is of degree lower than 6. But there is no such $G$-invariant curve not contained in $Q$. The linear system $|3H_S-C|$ also does not have any fixed point. Indeed, we have $(3H_S-C)^2=0$, so the curves in this linear system are disjoint. Hence, there is no base curve in $|3H_S-C|$ other than $C$ and it is nef. But $(3H_S-C)\cdot(3H-mC)<0$, which yields a contradiction.
    \end{enumerate}
\end{proof}

We exclude reducible curves in a similar way.
\begin{lemm}\label{lemma:redcurves not in Q}
    If $C$ is a reducible $G$-invariant curve of degree 12 not contained in $Q$, then each irreducible component of $C$ is not a non-canonical center of $(X,\lambda\cM_X)$.
\end{lemm}
\begin{proof}
   By Lemma \ref{lemm:12line}, $C$ is the union of one of the two orbits $\cC_6$ and $\cC_6'$ of 6 conics. Note that $\cC_6$ and $\cC_6'$ are exchanged by the $\fS_5$-permutation action. Without loss of generality, assume that $C$ is the union of conics in $\cC_6$ and components of $C$ are non-canonical centers of $(X,\lambda\cM_X)$.
   Let $C'$ be the union of conics in $\cC_6'$. By Remark~\ref{rema:6conicS}, $C\cup C'$ is contained in the smooth K3 surface $S=S_{2,25}\in\cP$ under the notation of Lemma \ref{lemm: sing S}.  Let $H_S$ be a general hyperplane on $S$. Similarly as in Lemma~\ref{lemma:irrcurves not in Q}, we know that 
   $$
   3H_S-mC
   $$
   is an effective divisor for some $m>1$.  Using equations, we find that $C+C'\sim_\dq4H_S$. Note that $C'$ is on the border of the Mori cone of $S$, since it is the disjoint union of six conics where each of them has self-intersection $-2$. So $C'-H_S\sim_\dq3H_S-C$ is not pseudo-effective, which contradicts the effectiveness of $3H_S-mC$.
\end{proof}

%\begin{prop}
  % Let $\Sigma$ be the non-canonical locus of the log pair $(X,\lambda\cM_X)$. Assume that $\Sigma$ is a finite set of points, then $\Sigma\subset\Sigma_5\cup\Sigma_5'.$
%\end{prop}
\subsection{Points outside $Q$}

\begin{lemm}\label{prop:pointsnotinQ}
	Let $P\in X$ and $\Sigma$ be its $G$-orbit. If $P\notin Q$ and $|\Sigma|\ne 5$, then $P$ is not a center of non-canonical singularities of  $(X,\lambda\mls_X)$.
\end{lemm}
	
\begin{proof}
   Assume that $P$ is a non-canonical center of $(X,\lambda\cM_X)$. We consider two cases:

    {\bf Case 1}: When $|\Sigma|\geq 20$. Remark~\ref{remark:Ziquan} implies that $(X,\frac32\lambda\cM_X)$ is not log-canonical at $P$. Let $\Lambda$ be the non-log-canonical locus of $(X,\frac32\lambda\cM_X)$, and $\Lambda_0$ its zero-dimensional component.
    
    Assume that a $G$-invariant curve $C$ is contained in $\Lambda$. Consider two general elements $M_1, M_2\in\cM_X$, we have that
    $$
    \frac 94\lambda^2(M_1\cdot M_2)=mC+\Delta,\quad m\geq(\mult_C(\frac32\lambda\cM_X))^2
    $$
 for an effective divisor $\Delta$ whose support does not contain $C.$ Intersecting with a general hyperplane section $H$ on $X,$ we obtain that
$$
\frac{81}{2}=H\cdot  \frac 94\lambda^2(M_1\cdot M_2)\geq m\deg(C).
$$
By Theorem~\ref{theo:4n2}, we know that $m>4$ and it follows that 
$$
\deg(C)\leq 10.
$$
Lemma~\ref{lemma:quadric-curves in S2} implies that $C\subset Q$. By assumption, we have $\Sigma\not\subset Q.$ Therefore, we know that $\Sigma\subset\Lambda_0$.
%and is either the Bring curve, or a curve of bidegree $(4,4)$ with genus 9. {\bf then what?????}

Let $\cI=\cI(X,\frac32\lambda\cM_X)$ be the multiplier ideal sheaf of $\frac32\lambda\cM_X$ on $X$.
Note that 
$$
K_X+\frac32\cM_X+\frac12\cO_X(1)\sim_\bQ \cO_X(2).
$$
Then, by Nadel vanishing theorem (cf. Theorem \ref{theo:nadel}), we know that $h^1(X,\cI\otimes\cO_X(2))=0$ and it follows that
$$
20\leq|\Sigma|\leq|\mathrm{Supp}(\cI)|\leq h^0(\cO_X(2))=14,
$$
which is a contradiction.

%{\bf Case 2}: If $\Sigma$ contains an orbit of length 12, the orbit is in $Q$. This case is excluded by Lemma~\ref{lemma:pts in S2}.
%WLOG, assume that $\Sigma\supset\Sigma_{12}$. Since $\Sigma_{12}\subset Q,$ we know that $(X,\lambda\cM+Q)$ is not log-canonical at $\Sigma_{12}.$ Let $\cM_Q=\cM_X\vert_{Q}$. By the inversion of adjunction, $(Q,\lambda\cM_Q)$ is also not log-canonical at $\Sigma_{12}$. Recall that $Q=\bP^1\times\bP^1$. We have 
%$$
%\lambda\cM_Q\sim_\bQ \cO_{Q}(3,3).
%$$
%Assume that $(Q,\lambda\cM_Q)$ is not log-canonical along a curve $C.$ Similarly as in case 1,  intersecting a general hyperplane section with a general elements in $\lambda\cM_Q$, we see that $\deg(C)<6$. By Lemma~\ref{curves not in Q} and~\ref{lemma:quadric-curves in S2}, such $C$ does not exist. So $(Q,\lambda\cM_Q)$ is not log-canonical at $n$ points, where $n\geq 12.$

%Let $\varepsilon$ be a positive rational number such that $(Q,(1-\varepsilon)\lambda\cM_Q)$ is not klt. Let $\cI$ be the multiplier ideal sheaf of $(1-\varepsilon)\lambda\cM_Q.$ We know $|\mathrm{Supp}(\cI)|\geq 12.$ Observe that 
%$$
%K_Q+(1-\varepsilon)\lambda\cM_Q+3\varepsilon\cO_{Q}(1,1)\sim_\bQ \cO_{Q}(1,1).
%$$
%Apply Nadel vanishing theorem, we obtain 
%$$
%12\leq|\mathrm{Supp}(\cI)|\leq h^0(\cO_Q(1,1))=4,
%$$
%which is absurd.

{\bf Case 2:} When $|\Sigma|<20$, then by the classification of orbits we know that $|\Sigma|=10$. This case is excluded by \cite[Proof of Proposition 3.4]{CSZ}. The proof there applies verbatim.
\end{proof}

\subsection{Points inside $Q$}

Here we finish the proof of Proposition~\ref{prop:main1} by finding the $G$-equivariant $\alpha$-invariant of $Q$. 
\begin{lemm}\label{lemm:alphaQ}
 One has $\alpha_G(Q)=\frac32$.    
\end{lemm}
\begin{proof}
    By Lemma~\ref{lemma:quadric-curves in S2}, we see that the Bring curve $B_6$ of bidegree $(3,3)$ is the $G$-invariant divisor in $Q$ with the least degree. By definition of the $\alpha$-invariant, we have $\alpha_G(Q)\leq \frac32$. Assume that $\alpha_G(Q)<\frac32.$ Then there exists a $G$-invariant effective $\bQ$-divisor $D$ on $S$ such that 
    $$
    D\sim_Q \cO_Q(3,3)
    $$
    and $(Q,D)$ is not log-canonical. Let $\Lambda$ be the non-log-canonical locus of $(Q,D)$. Assume that $\Lambda$ contains a curve $C\subset Q$. We have $D=mC+\Delta$ where $m>1$ and $\Delta$ is an effective divisor whose support does not contain $C$. Intersecting with a general hyperplane section $H$, we obtain 
$$
6=H\cdot D\geq m\deg(C).
$$
It follows that $\deg(C)<6$. By Lemma~\ref{lemma:quadric-curves in S2}, such curves do not exist.

Thus, $\Lambda$ is 0-dimensional. We have $|\Lambda|\geq12$ since orbits of length 5 and 10 are not in $Q$.  Let $\varepsilon\in\bQ^{>0}$ such that $(Q,(1-\varepsilon)D)$ is not klt at points in $\Lambda$, and $\cI$ the multiplier ideal sheaf of $(1-\varepsilon)D.$  Note that 
$$
K_Q+(1-\varepsilon)D+3\varepsilon\cO_{Q}(1,1)\sim_\bQ \cO_{Q}(1,1).
$$
Applying Nadel vanishing theorem (cf. Theorem \ref{theo:nadel}), we obtain 
$$
12\leq|\mathrm{Supp}(\cI)|\leq h^0(\cO_Q(1,1))=4,
$$
which is absurd. So we obtain a contradiction and $\alpha_G(Q)=\frac32$.
\end{proof}

\begin{coro}\label{coro:canonicalinQ}
    Let $Z$ be a non-canonical center of the pair $(X,\lambda\cM_X)$, then $Z\not\subset Q$.
\end{coro}
\begin{proof}
 If $Z$ is contained in $Q$, then by inversion of adjunction, $Z$ is a non-log-canonical center of $(Q,\lambda\cM_X\vert_Q)$, which contradicts Lemma~\ref{lemm:alphaQ}.
\end{proof}

\section{The standard $\fA_5$-action on the cubic threefold
}\label{sect:a5cub1}

In this section, we study the cubic threefold $Y\subset\bP^4_{x_1,\ldots,x_4}$ given by
\begin{multline*}
x_1x_2x_3+x_1x_2x_4+x_1x_2x_5+x_1x_3x_4+x_1x_3x_5+x_1x_4x_5+\\+x_2x_3x_4+x_2x_3x_5+x_2x_4x_5+x_3x_4x_5=0
\end{multline*}
with the same $G=\fA_5$-action through permutations of coordinates. Note that $\Sing(Y)$ consists of 5 nodes. The aim of this section is to prove the following result.
\begin{prop}\label{prop:main2}
    Let $\mls_Y$ be a non-empty mobile $G$-invariant linear system on $Y$, and let $\mu\in\dq$ such that $\mu\mls_Y\sim_\dq-K_Y$. Then the log pair $(Y,\mu\mls_Y)$ is canonical away from $\sing(Y)$.
\end{prop}
\begin{proof}
    This follows from Propositions~\ref{prop:lcYiscurve} and \ref{prop:pointsoutsideR}, and Corollary \ref{coro:lcYinR}.
\end{proof}

\begin{rema}\label{rema:degreebound2}
    If a curve $C$ is a center of non-canonical singularities, then for any two general members $M_1,M_2\in\cM_Y$, we have that
$$
\lambda^2(M_1\cdot M_2)=mC+\Delta
$$
for some $m>1$ and some effective divisor $\Delta$ not supported along $C$. Intersecting with a general hyperplane $H$, we obtain that 
\begin{align}\label{eqn:bound12}
    12=\lambda^2(M_1\cdot M_2\cdot H)>\deg(C).
\end{align}
\end{rema}

Thus, we need to consider $G$-orbits of lengths less than 20 and $G$-invariant curves of degrees lower than 12.
As in the previous section, we split into subsections according to whether or not a potential non-canonical center of $(Y,\mu\mls_Y)$ belongs to the $G$-invariant hyperplane section. 

%In the second part, we study the canonicity of $(Y,\mu\mls_Y)$ at the $G$-orbits outside $R$, then along the invariant curves outside $R$, and in the fourth subsection we prove that none of the curves or points in $R$ can be a center of non-canonical singularities of $(Y,\mu\mls_Y)$.

\subsection{Small $G$-orbits and $G$-invariant curves of low degrees}
We begin with identifying small $G$-orbits and $G$-invariant curves of low degrees in $Y$. 
\begin{lemm}
\label{lemm:A5orbitcub}
  A $G$-orbit of points in $Y$ with length $<$ 20 is one of the following:
\begin{align*}
    \Sigma_5^1&=\text{the orbit of}\quad[1:0:0:0:0],\\
     \Sigma_5^2&=\text{the orbit of} \quad[-2:3:3:3:3],\\
       \Sigma^1_{10}&=\text{the orbit of} \quad[1:1:0:0:0],\\ 
        \Sigma^2_{10}&=\text{the orbit of} \quad[1:-1:0:0:0],\\ 
          \Sigma^3_{10}&=\text{the orbit of} \quad[-6-2\sqrt6:-6-2\sqrt6:6:6:6],\\ 
            \Sigma^4_{10}&=\text{the orbit of} \quad[-6+2\sqrt6:-6+2\sqrt6:6:6:6],\\ 
     \Sigma_{12}^1&=\text{the orbit of} \quad[1:\zeta_5:\zeta_5^2:\zeta_5^3:\zeta_5^4],\\
        \Sigma_{12}^2&=\text{the orbit of}\quad [1:\zeta_5^2:\zeta_5^4:\zeta_5:\zeta_5^3],\\
         \Sigma_{15}&=\text{the orbit of}\quad [0:-1:-1:1:1],
\end{align*}
where the length of each orbit is indicated by the subscript.
\end{lemm}

With the notation above, $\Sing(Y)=\Sigma_5^1$. There is a unique $G$-invariant hyperplane section in $Y$, given by 
$$
R:=\{x_1+x_2+x_3+x_4+x_5=0\}\cap Y.
$$
Note that $R$ is the Clebsch cubic surface. One can check that
\begin{align}\label{eqn:orbRcub}
    \Sigma_{10}^2,\Sigma^{1}_{12},\Sigma^{2}_{12},\Sigma_{15}\in R,\quad \Sigma^1_{5},\Sigma_5^2,\Sigma_{10}^{1},\Sigma_{10}^{3},\Sigma_{10}^{4}\not\in R.
\end{align}

We recall some facts about the $\fA_5$-equivariant geometry of $R$, see \cite[Section 6.3]{CheltsovShramov} for more details. 
The surface $R$ is $G$-linearizable. Indeed, there are two unions $L_6, L_6'$ of 6 pairwise disjoint lines in $R$. Respective contractions of $L_6$ and $L_6'$ give two $G$-equivariantly birational maps $\pi,\pi':R\to \bP^2$.
There is a unique $G$-invariant conic in $\bP^2$. We denote its strict transforms under $\pi$ and $\pi'$ by $C_6$ and $C_6'$ respectively. 
\begin{lemm}\label{lemm:A5invcurveClebsch}
    Let $C$ be a $G$-invariant curve in $Y$ with $\mathrm{deg}(C)<10$. Then $C\subset R$, $\deg(C)=6$, and $C$ is one of the following 
  $$
  L_6, L_6', C_6, C_6',\quad\text{or the Bring curve }B_6\text{ defined by \eqref{bring}}.
  $$
\end{lemm}
\begin{proof}
    If $C\not\subset R$, then $C\cdot R=\deg(C)<10$. By \eqref{eqn:orbRcub}, we know that this is impossible. Thus $C\subset R$. The rest of the lemma follows from \cite[Theorem 6.3.18]{CheltsovShramov}.
\end{proof}

\begin{lemm}\label{lemm:10line}
    Let $C$ be a $G$-invariant curve in $Y$, of degree 10 and not contained in $R$. Then $C$ is the union of 10 lines in the $G$-orbit of 
  $$
        \{x_3=x_4=x_5=0\}\subset Y.
$$
Moreover, these lines are the lines that pass through pairs of points in the singular locus of $Y$.
\end{lemm}
\begin{proof}
     We may assume that $C$ is $G$-irreducible. When $C$ is an irreducible curve, by computation, we check that there is no $G$-invariant irreducible curve with a generic stabilizer. So $G$ acts faithfully on $C$. Note that $C\cdot R=\deg(C)=10$. By \eqref{eqn:orbRcub}, we see that $C\cap R=\Sigma_{10}^2$ where all 10 points are smooth points of $C$. The stabilizer of a point in $C\cap R$ is $\fS_3$, which is a contradiction, since it should act faithfully in the tangent space of $C$ at this point.  It follows that $C$ is a reducible curve.
    
    So $C$ can be 5 conics or 10 lines. Assume that $C$ consists of 5 conics. Each conic spans a plane in $\bP^4$, left invariant by $\fA_4\subset G$. Each such plane intersects $X$ along the conic and a residual line. Therefore, we obtain a $G$-orbit of 5 lines. One can check that there is no such orbit of lines in $X$. Similarly, we find that there is only one $G$-orbit of 10 lines, as is given in the assertion.

\end{proof}

\subsection{Invariant curves not contained in $R$}
With the classification of $G$-irreducible invariant curves, we exclude curves not contained in $R$ as non-canonical centers in this case.
\begin{prop}\label{prop:lcYiscurve}
    Let $C$ be a $G$-invariant curve in $Y$ not contained in $R$. Then each irreducible component of $C$ is not a non-canonical center of the pair $(Y,\mu\cM_Y)$.
\end{prop}
\begin{proof}
Assume that the irreducible components of $C$ are non-canonical centers. By Remark \ref{rema:degreebound2}, we have $\deg(C)<12$. From \eqref{eqn:orbRcub}, we see that 
    $$
    \deg(C)=10
    $$
 and $C$ is the union of 10 lines given in Lemma~\ref{lemm:10line}. This is impossible by \cite[Proof of Proposition 3.5]{CSZ}. 
    \end{proof}

\subsection{Points outside $R$}

\begin{prop}\label{prop:pointsoutsideR}
    Let $P$ be a point outside $R$, and $\Sigma$ its $G$-orbit. If $\Sigma\ne \Sigma^1_5$, then $P$ is not a non-canonical center of $(Y,\mu\mls_Y)$.
\end{prop}
\begin{proof}
Assume that $P$ is a non-canoncial center of $(Y,\mu\mls_Y)$. By Remark~\ref{remark:Ziquan}, we know that $P$ is a non-log-canonical center of $(Y,\frac32\mu\cM_Y)$. Let $\varepsilon$ be a positive rational number such that 
    $$
    \Sigma\subset\Omega,\quad \Omega:=\mathrm{Nklt}(Y,(\frac{3}{2}-\varepsilon)\mu\cM_Y)
    $$
    where $\Omega$ is the non-klt locus of $(Y,(\frac{3}{2}-\varepsilon)\mu\cM_Y)$. 
    
 Assume that there is a curve $C\subset \Omega$. Let $M_1,M_2\in(\frac{3}{2}-\varepsilon)\mu\cM_Y$ and $H$ a general hyperplane section of $Y$. Similarly as before, we have 
    $$
    27\geq H\cdot(\frac{3}{2}-\varepsilon)^2\mu^2(M_1\cdot M_2)\geq m\deg(C)>4\deg(C)
    $$
for some number $m>4$ by Theorem~\ref{theo:4n2}. Lemma~\ref{lemm:A5invcurveClebsch} implies that $\deg(C)=6$ and $C\subset R$. This shows that every curve in $\Omega$ is in $R$. It follows that the 0-dimensional component $\Omega_0$ of $\Omega$ is non-empty since $P\not\in R$. In particular, $\Omega_0\supset\Sigma$. Observe that 
$$
K_Y+(\frac{3}{2}-\varepsilon)\mu\cM_Y+2\varepsilon\cO_Y(1)\sim_\bQ \cO_Y(1).
$$
Let $\cI$ be the multiplier ideal sheaf of $(\frac{3}{2}-\varepsilon)\mu\cM_Y$. By Nadel vanishing theorem (Theorem \ref{theo:nadel}), we have 
$
h^1(\cI\otimes\cO_Y(1))=0.
$ This implies that 
$$
|\Omega_0|\leq h^0(\cO_Y(1))=5.
$$
It follows that $\Omega_0=\Sigma=\Sigma_5^1$ or $\Sigma_5^2$. The latter is impossible by \cite[Proposition 3.5]{CSZ}.
\end{proof}

\subsection{Points inside $R$}

Similarly as in the previous section,
it suffices to find the $G$-equivariant $\alpha$-invariant of $R$.
\begin{lemm}\label{lemm:alphaR}
    One has $\alpha_G(R)=2$.
\end{lemm}
\begin{proof}
    Note that $B_6\subset R$ is a $G$-invariant effective divisor such that $B_6\sim_\bQ -2K_R$. It follows that $\alpha_G(R)\leq 2$. Suppose that $\alpha_G(R)<2$. Then there exists a $G$-invariant effective $\dq$-divisor $D\sim_\dq-2K_R$ such that $(R,D)$ is not log-canonical. Let $\Lambda$ be the non-log-canonical locus of $(R,D)$. Let $\varepsilon\in\bQ_{>0}$ such that the non-klt locus $\Omega$ of $(R,(1-\varepsilon)D)$ contains  $\Lambda$.  Assume that $\Omega$ contains some curve  $C'$, then 
$$
(1-\varepsilon)D=m C'+\Delta,\quad m\geq 1
$$
for some effective $1$-cycle $\Delta$ whose support does not contain $C'$. Intersecting with a general hyperplane section $H$ on $R$, we obtain 
$$
6>H\cdot(1-\varepsilon)D=H\cdot(mC'+\Delta)\geq\deg(C'),
$$
which is impossible by Lemma~\ref{lemm:A5invcurveClebsch}.

Thus, $\Omega$ consists of finitely many points. Let $n=|\Lambda|$ and $\cI$ be the multiplier ideal sheaf of $(1-\varepsilon)\mu\cM_R$. Observe that 
$$
K_R+(1-\varepsilon)D+2\varepsilon\cO_R(1)\sim_\bQ \cO_R(1).
$$
By Nadel vanishing theorem, we know that 
$h^1(\cO_R(1)\otimes\cI)=0$ and 
$$
n=|\Lambda|\leq|\Omega|\leq h^0(\cO_R(1))=4,
$$
which implies that $n=0$ since there is no $G$-orbit of length $\leq 4$ in $R$. 
\end{proof}

\begin{coro}\label{coro:lcYinR}
   Let $Z$ be a non-canonical center of the pair $(Y, \mu\cM_Y)$, then $Z\not\subset R$.
\end{coro}
\begin{proof}
    Assume that $Z$ is contained in $R$. By inversion of adjunction, the pair $(R,\mu\cM_X\vert_R)$ is not log-canonical, which contradicts Lemma~\ref{lemm:alphaR}.
\end{proof}

\section{The nonstandard $\fA_5$-action on the quadric threefold}\label{sect:a5quad2}
In this section, we study the non-standard $\fA_5$-action. Let $G=\fA_5$ acting on the smooth quadric threefold given by
\begin{equation}\label{eqquad2}
X=\left\{\sum_{1\leq i\leq j\leq 5}x_ix_j=0\right\}\subset\bP^4
\end{equation}
with the $G$-action generated by 
\begin{align}\label{eqn:2a5gen}
  (\mathbf x)&\mapsto (x_4,x_1,x_5,x_2,-x_1-x_2-x_3-x_4-x_5), \notag \\
    (\mathbf x)&\mapsto (x_4,-x_1-x_2-x_3-x_4-x_5,x_1,x_3,x_2). 
\end{align}

The aim of this section is to prove the following result.

\begin{prop}\label{prop:main3}
    Let $\mls_X$ be a non-empty mobile $G$-invariant linear system on $X$, and $\lambda\in\dq$ such that $\lambda\mls_X\sim_\dq-K_X$. Let $Z$ be a $G$-irreducible subvariety whose components are centers of non-canonical singularities of $(X,\lambda\mls_X)$. Then $Z$ is one of the following: 
    \begin{itemize}
        \item the union of 5 points in the orbit $\Sigma_5$ or $\Sigma_5'$ given in Lemma~\ref{lemm:2a5orbit},
        \item the rational curve $C_4$ or $C_4'$ of degree 4 given by \eqref{not:deg4},
        \item the rational curve $C_8$ or $C_8'$ of degree 8 described in Remark~\ref{rema:c8curve},
        \item (possibly) a smooth irreducible curve of degree 10 and genus 6.
    \end{itemize}
\end{prop}
\begin{proof}
    This follows from Proposition~\ref{prop:curves2A5X} and Proposition~\ref{prop:ptcase2}.
\end{proof}

\begin{rema}
    Equations of $C_4, C_4', C_8, C_8'$ can be found in \cite{Z-web}. Sarkisov links centered at these curves are presented in later subsections. We do not know the existence of the curve of degree 10. This is not necessary for our main result, see Lemma~\ref{lemm:linkC8C10}.
\end{rema}
Note that the $G$-action on the ambient $\bP^4$ arises from the unique 5-dimensional irreducible linear representation of $G$. The nature of this action creates more challenges for our classifications since there are more possibilities of $G$-invariant curves and non-canonical centers, cf. Proposition~\ref{prop:main1}. The $\fA_5$-equivariant geometry of K3 surfaces turns out to be crucial to our analysis in this section.

First, we classify $G$-orbits of lengths less than 20, and the $G$-irreducible invariant curves of degrees at most 17. In the second subsection, we will study the singularities of pairs $(X,\lambda\mls_X)$ along $G$-invariant curves, and in the third subsection, we will study them along $G$-orbits. %Throughout this section, $\mls_X$ is a non-empty mobile $G$-invariant linear system on $X$, and $\lambda\in\dq$ such that $\lambda\mls_X\sim_\mathbb Q-K_X$.

\subsection{Small $G$-orbits and $G$-invariant curves of low degrees}
\begin{lemm}\label{lemm:2a5orbit}
      A $G$-orbit of points in $X$ with length $<20$ is one of the following:
\begin{align*}
    \Sigma_5&=\text{the orbit of}\quad[1:\zeta_6-1:-\zeta_6:\zeta_6-1:1],\\
     \Sigma_5'&=\text{the orbit of} \quad[1:-\zeta_6:\zeta_6-1:-\zeta_6:1],\\
     \Sigma_{12}&=\text{the orbit of} \quad[\zeta_5^3:\zeta_5^2:0:\zeta_5:1],\\
    \Sigma_{12}'&=\text{the orbit of}\quad [\zeta_5^4:\zeta_5:0:\zeta_5^3:1],
\end{align*}
where the length of each orbit is indicated by the subscript.
\end{lemm}

\begin{lemm}\label{lemm:genericstabcurve-A5quadric2}
    Let $C$ be a $G$-invariant curve in $X$ with $\deg(C)\leq17$. Then $C$ has trivial generic stabilizer, i.e, the $G$-orbit of a general point in $C$ has length 60.
\end{lemm}
\begin{proof}
By direct computation, one sees that the only irreducible curves in $X$ with non-trivial generic stabilizers are lines and conics, and their $G$-orbits have lengths 20 and 15 respectively. 
\end{proof}

\begin{lemm}\label{lemm: G-inv red curves}
    Let $C$ be a $G$-invariant reducible curve of degree at most 17. Then $C$ is the union of curves in one of the following orbits:
    \begin{itemize}
        \item 
        an orbit of 5 conics 
        $$
        \mathcal{C}_{5}=\text{orbit of }\{x_1 + x_4=
x_2 + x_3=0\}\cap X,
$$
        \item 
        one of the following 2 orbits of 6 conics 
        $$
        \mathcal{C}_{6}=\text{orbit of } C_1,\qquad 
           \mathcal{C}_{6}'=\text{orbit of }C_2, $$
        where 
        $$
        a=\zeta_5+\zeta_5^4,
        $$
        $$
        C_1=\{x_1 - x_3 -a(x_4 -x_5)=x_2 + a(x_3-x_4) - x_5=0\}\cap X,
        $$
          $$
        C_2=\{x_1 - x_3 +(1+a)(x_4 -x_5)=x_2 -(1+a)(x_3-x_4) - x_5=0\}\cap X,
        $$
        \item 
    one of the following 2 orbits of 12 lines
     $$
          \mathcal L_{12}=\text{the orbit of}\quad\{\sum_{i=1}^5x_i=
\sum_{i=1}^5\zeta_5^{i-1}x_i=\sum_{i=1}^5\zeta_5^{3(i-1)}x_i=0\},
   $$
$$
    \mathcal L_{12}'=\text{the orbit of}\quad\{\sum_{i=1}^5x_i=
\sum_{i=1}^5\zeta_5^{i-1}x_ix_i=\sum_{i=1}^5\zeta_5^{2(i-1)}x_i=0\}.
$$
    \end{itemize}
    Each of the orbits above consists of pairwise disjoint curves.
\end{lemm}
\begin{proof}
The proof is similar to that of Lemma~\ref{lemm:12line}.    
\end{proof}

\subsection{Invariant curves}
Similarly as in Section~\ref{sect:a5quad1} (see Remark \ref{rema:degreebound1}), if a curve $C$ is in the non-canonical center of $(X,\lambda\cM_X)$, we have 
$$
\deg(C)\leq 17.
$$

The main result of this subsection is:

\begin{prop}\label{prop:curves2A5X}
    Let $\mls_X$ be a non-empty mobile $G$-invariant linear system on $X$, and $\lambda\in\bQ$ such that $\lambda\mls_X\sim_\dq-K_X$. If a curve $C$ is a non-canonical center of $(X,\lambda\mls_X)$, then $C$ is an irreducible curve of degree 4,8, or 10, and is one of the curves described in Proposition~\ref{prop:main3}. 
\end{prop}
\begin{proof}
  We explain how the results in this subsection show the assertion. Lemma~\ref{lemm:quad2curveinR} shows that if $C$ is contained in certain surfaces $R$ or $R'$ explicitly given by \eqref{not:RR'}, then $\deg(C)=4$. When $C$ is not in $R$ or $R'$, Lemma~\ref{lemma:curves not in non normal K3} shows $\deg(C)\in\{8,10,12,16\}$. Lemma~\ref{lemm:deg12} excludes the case $\deg(C)=12$. Lemmas~\ref{lemm:deg16smooth} and~\ref{lemm:deg16sing} show that $\deg(C)=16$ is also impossible. Then Lemmas~\ref{lemm:linkC8C10},~\ref{lemm:linkC4},~\ref{lemm:quad2curveinR} and~\ref{lemm:allcurve80} prove that all such curves are among those described in Proposition~\ref{prop:main3}.
\end{proof}

First, we present curves of degree 4. Consider the pencil consisting of $G$-invariant cubics in $\bP^4$ given by 
$$
\{a_1f_1+a_2f_2=0\}\subset\bP^4_{x_1,\ldots,x_5},\quad [a_1:a_2]\in\bP^1
$$
where
\begin{multline}\label{eqn:2a5cubicpencil}
    f_1=x_1^2x_2 + x_1x_2^2 + 2x_1x_2x_3 + x_2^2x_3 + x_2x_3^2 + 2x_2x_3x_4 + x_3^2x_4 + 
        x_3x_4^2 +\\+ x_1^2x_5 + 2x_1x_2x_5 + 2x_1x_4x_5 + 2x_3x_4x_5 + x_4^2x_5 + 
        x_1x_5^2 + x_4x_5^2,
\end{multline}
\begin{multline*}
    f_2=x_1^2x_3 + x_1x_3^2 + x_1^2x_4 + 2x_1x_2x_4 + x_2^2x_4 + 2x_1x_3x_4 + x_1x_4^2 + 
        x_2x_4^2 + \\+x_2^2x_5 + 2x_1x_3x_5 + 2x_2x_3x_5 + x_3^2x_5 + 2x_2x_4x_5 + 
        x_2x_5^2 + x_3x_5^2.
\end{multline*}
In particular, there are two $G$-invariant chordal cubics in $\bP^4$, i.e., the cubic threefold whose singular locus is a twisted quartic curve. Their intersections with $X$ are two non-normal surfaces, given by 
\begin{align}
  R&=\{(-\zeta_5^3 - \zeta_5^2 + 1)f_1+f_2=0\}\cap X,\label{not:RR'}\\
    R'&=\{(\zeta_5^3 + \zeta_5^2 + 2)f_1+f_2=0\}\cap X.\nonumber
\end{align}
Their intersection
$
R\cap R'
$ is an irreducible curve whose singular locus is $\Sigma_{12}\cup\Sigma_{12}'$.
Let  
\begin{align}
    C_4=\Sing(R),\quad C'_4=\Sing(R').\label{not:deg4}
\end{align}
Then $ C_4$ and $ C'_4$ are quartic rational normal curves such that 
$$
\Sigma_{12}\in  C_4\setminus C'_4,\quad \Sigma_{12}'\in C'_4\setminus C_4.
$$
These two curves can be non-canonical centers of $(X,\lambda\cM_X)$. Sarkisov links centered at them are involutions on $X$, presented in Lemma~\ref{lemm:linkC4}.
Now, let $\mathcal P$ be the pencil consisting of $G$-invariant K3 surfaces $S_{a_1,a_2}$ in $X$ given by  
\begin{align}\label{eqn:SA1A2}
    S_{a_1,a_2}:=\{a_1f_1+a_2f_2=0\}\cap X,\qquad [a_1:a_2]\in\bP^1. 
\end{align}
\begin{rema}\label{rk: K3s containing reducible curves}
Each of the orbits in Lemma \ref{lemm: G-inv red curves} is contained in a unique member in $\cP$. We record 
$$
\begin{tabular}{c|c|c}
     \text{orbit}& $a_1$&$a_2$  \\\hline
     $\mathcal C_5$&1&1\\\hline
    $\mathcal C_6$&$63(\zeta_5^3 + \zeta_5^2) + 145$&89\\\hline
    $\mathcal C_6'$&$63(\zeta_5^3 + \zeta_5^2) + 145$&89\\\hline
      $\mathcal L_{12}$&$-\zeta_5^3 - \zeta_5^2 + 1$&1\\\hline
      $\mathcal L_{12}'$&$\zeta_5^3 + \zeta_5^2 + 2$&1\\
\end{tabular}
$$
where each orbit in the first column is contained in $S_{a_1,a_2}\in\cP$ for $a_1,a_2$ indicated in the same row.
    %With the notations of Lemma \ref{lemm: G-inv red curves}, we have that $\mathcal C_5$ is contained in $S_{1,1}$, and $\mathcal C_6,\mathcal C_6'$ are contained in $S_{\zeta,\xi},$ where $\zeta=63(\zeta_5^3 + \zeta_5^2) + 145$ and $\xi=89$. On the other hand, $\mathcal L_{12}$ is contained in $\mathcal R$ and $\mathcal L_{12}'$ is contained in $\mathcal R'$.
\end{rema}
\begin{lemm}\label{lemm: sing S 2}
    A surface $S_{a_1,a_2}\in \cP$  is singular if and only if 
    \begin{enumerate}
    \item $[a_1:a_2]=[-3\zeta_6 + 8:7],\quad$ $\Sing(S_{a_1,a_2})=\Sigma_5$,
    \item $[a_1:a_2]=[3\zeta_6 + 5:7],\quad $  $\Sing(S_{a_1,a_2})=\Sigma_5'$,
    \item $[a_1:a_2]=[-\zeta_5^3 - \zeta_5^2 + 1: 1],\quad $ $\Sing(S_{a_1,a_2})=  C_4$, $S_{a_1,a_2}=R$,
    \item  $[a_1:a_2]=[\zeta_5^3 + \zeta_5^2 + 2: 1],\quad $ $\Sing(S_{a_1,a_2})= C_4'$, $S_{a_1,a_2}=R'$.
    \end{enumerate}
Moreover, when $S_{a_1,a_2}\in\cP$ is smooth, it contains no orbit of length 5.
\end{lemm}
\begin{proof}

First, we consider the case when $S=S_{a_1,a_2}$ is normal. The singular locus $\Sing(S_{a_1,a_2})$ forms a $G$-invariant set. Note that $K_S\sim0$. If $S$ has non-du Val singularities, it has at least 12 of them, since the smallest orbit on $X$ has length 12. This is impossible. Thus, $S$ has at worst du Val singularities and its minimal resolution is a smooth K3 surface, whose Picard rank is bounded by 20. Then $|\Sing(S_{a_1,a_2})|< 20$ and $\Sing(S_{a_1,a_2})$ consists of orbits in Lemma~\ref{lemm:2a5orbit}. One can check that the four cases in the assertion are the only possible cases. Note that being singular along $\Sigma_{12}$ or $\Sigma_{12}'$ forces $S_{a_1,a_2}$ to be non-normal.
 
 Now assume that $S=S_{a_1,a_2}$ is singular along a curve $Z$. Let $S'$ be a general member of $\cP$. Recall that $S'\cap S$ is an irreducible curve whose singular locus is $\Sigma_{12}\cup\Sigma_{12}'$. It follows that $\emptyset\ne Z\cap S\subset\Sigma_{12}\cup\Sigma_{12}'$ and the only possible cases are $S=R$ or $R'$.
\end{proof}

%\begin{lemm}
%The log pair $(X,\lambda\cM_X)$ is canonical at any points in $R\cup R'$.
%\end{lemm}
%\begin{proof}
 %   By symmetry and the inversion of adjunction, it suffices to show $(R,\lambda\cM_R)$ is lc, where $\cM_R=\cM_X\vert_R$. Let $\Lambda$ be the non-lc locus of  $(R,\lambda\cM_R)$.
  %  Assume that $\Lambda\supset C$ for some curve $C$, i.e., 
   % $$
    %\lambda\cM_S=mC+\Delta
    %$$
    %where $m>1$ and $\Delta$ is an effective divisor not supported along $C$. 
    %Recall that $R$ is $G$-isomorphic?? to the quadric surface $E$. Then $\tilde\varphi(\tilde H\vert_{\tilde R})=s+3f$ ?????, where $\tilde H$ is the pullback in $\tilde X$ of the class of a general hyperplane section $H$ on $X$. Assume that $\tilde C=\tilde \varphi(\pi^{-1}(C))=as+bf$ for some $a,b>0$.
    %....
    %Then we have 
    %$$
    %\deg(C)=C\cdot H=(as+bf)(s+3f)=3a+b\leq 17.
    %$$
    %On the other hand, since the $G$-action on $E$ leaves invariant a rational curve, we know $G=\fA_5$ acts diagonally on $E=\bP^1\times\bP^1$. A direct computation leads us to the following possibilities:
    %$$
    %(a,b)=(1,1),(2,2),(3,3),(4,4),(1,11),(1,13),(2,10).
    %$$
    %Also note that 
    %$$
    %\tilde\varphi(\pi^{-1}(\lambda\cM_S))=\tilde\varphi(\pi^{-1}(3H\vert_R))=3s+9f.
    %$$
    %We see that the only possibility is $(a,b)=(1,1)$ or $(2,2)$, i.e., $C=?$
%\end{proof}

\begin{lemm}\label{lemma:curves not in non normal K3}
    Let $C$ be a $G$-invariant curve not contained in $R\cup R'$ such that $\deg(C)\leq 17$. Then the following assertions hold.
    \begin{enumerate}
        \item We have $\deg(C)\in\{8,10,12,16\}$\label{lemma:curves not in non normal K3 - 1}.
        \item The irreducible components of $C$ are pairwise disjoint.
        \item In the pencil $\mathcal P$, there exists a unique surface $S$ containing $C$.
        \item If $C$ is irreducible, then $C$ is a Cartier divisor on $S$.
        \item The surface $S$ is smooth.
        \item If $\deg(C)\ne 16,$ then $C$ is smooth.
        \item If $\deg(C)\ne 16$ and $C$ is irreducible, then its genus satisfies
        $$
        g(C)=\begin{cases}
            0&\text{if }\deg(C)=8\text{ or } 12,\\
         6&\text{if }\deg(C)=10.
        \end{cases}
        $$
    \end{enumerate}
\end{lemm}

\begin{proof}We may assume that $C$ is $G$-irreducible.
    \begin{enumerate}
    \item Since $C$ is not contained in $R$, we have $C\cdot R=3\deg(C)\le51$, and $C\cap R$ splits into $G$-orbits. Hence,
    $$
    C\cdot R=12a+20b+30c=3\deg(C)\le51,\quad a,b,c\in\bZ_{\geq0}.
    $$ 
    If $a>0$, since $R$ is singular at $\Sigma_{12}$, we  have $2\le a\le4$ and $b=c=0$. Then $\deg(C)\in\{8,12,16\}$. If $a=0$, then $b=0$ and $c=1$. In this case, $\deg(C)=10.$ %Note that in this case there is no $G$-orbit of length 12 on $C$, since all the $G$-orbits of length 12 on $X$ belong to $R$. In particular, it implies that the curve $C$ is not rational when $\deg(C)=10$.

        \item This is obvious if $C$ is irreducible. When $C$ is reducible, the assertion follows from the classification in Lemma~\ref{lemm: G-inv red curves}.
        \item Let $P$ be a general point of $C$. There exists a unique surface $S$ in $\mathcal P$ that passes through $P$. The intersection $S\cap C$ contains the orbit of $P$, which has length $60$ by Lemma \ref{lemm:genericstabcurve-A5quadric2}. If $C$ is not contained in $S$, we have that $C\cdot S=3\deg(C)\in\{24,30,36,48\}$, which is a contradiction. Therefore $C\subset S$.
        \item Assume that $C\cap\sing(S)\ne\emptyset$, otherwise $C$ is Cartier. Let $f\colon \widetilde S\rightarrow S$ be the blowup of $C\cap\Sing(S)$. Denoting by $\widetilde C$ the strict transform of the curve $C$ by $f$, we have $\widetilde C\sim_\dq f^*(C)-mE$ for some $m\in\frac{1}{2}\dz$, where $E$ is the exceptional divisor of $f$. Consider a point $P\in C\cap\sing(S)$, and a component $E_P$ of $E$ mapped to $P$. We have $\widetilde C\cdot E_P=2m$. By assumption, we have $S\ne R$ and $R'$. Lemma \ref{lemm: sing S 2} implies that $|C\cap\Sing(S)|=5$ and the stabilizer of $P$ is $\fA_4$, which acts faithfully on $E_P=\bP^1$. It follows that $2m=4a+6b+12c$, for some non-negative integers $a,b,c$, because the possible lengths of $\aq$-orbits on $\pl$ are 4,6, and 12. Hence $m$ is an integer, and $C$ is Cartier.
        \item When $C$ is reducible, this follows from Remark~\ref{rk: K3s containing reducible curves}. We assume that $C$ is irreducible. The proof is similar to that of Lemma \ref{lemma:curves not in Q}. Assume that $S$ is singular, then $\rk(\pic^G(S))=1$. 
        Let $H$ be a general hyperplane section on $X$, and $H_S$ its restriction to $S$. Since $\deg(H_S)=6$ is not a square, we know that $\Pic^G(S)$ is generated by $H_S$. Note that $C\in\Pic^G(S)$  because $C$ is Cartier. It follows that $C\sim nH_S$ for some integer $n$. We know that $\deg(C)=C\cdot H_S=6n$, which implies that $\deg(C)=12$ and $n=2$.
    Recall that the only $G$-invariant quadric hypersurface in $\bP^4$ is $X$. So there is no $G$-invariant curve linear equivalent to $2H_S$ and we obtain a contradiction.
        \item When $C$ is reducible, this follows from Lemma~\ref{lemm: G-inv red curves}. We assume that $C$ is irreducible and singular. Let $\widetilde C$ be a minimal resolution of singularities of $C$. Since $S$ is smooth, then by Lemma \ref{lemm: sing S 2}, we know that $S$ does not contain an orbit of length 5. We get 
        $$
        g(C)=g(\widetilde C)=p_a(\widetilde C)=p_a(C)-12a-20b-30c-60d,
        $$ where $a,b,c,d$ are non-negative integers which are not all 0, $g(C)$ is the geometric genus, and $p_a(C)$ is the arithmetic genus of $C$. Hodge index theorem gives 
        \begin{align}\label{eqnLhodge}
                   C^2\leq \frac{(C\cdot H_S)^2}{(H_S)^2}. 
        \end{align}
        If this is an equality, then $C\sim nH_S$, for some $n\in\dz$, and we have proved above
        that this is impossible. So \eqref{eqnLhodge} is a strict inequality, i.e., 
        $$
        C^2\leq\begin{cases}
            10&\text{ if } \deg(C)=8,\\
            16&\text{ if } \deg(C)=10,\\
            22&\text{ if } \deg(C)=12.
        \end{cases}
        $$
      Recall that $p_a(C)=\frac{C^2+2}{2}$. We obtain that
       \begin{align}\label{eqn:boundgenus}
        p_a(C)\leq\begin{cases}
            6&\text{if }\deg(C)=8,\\
           9&\text{if }\deg(C)=10,\\
            12&\text{if }\deg(C)=12.\\
        \end{cases}
    \end{align}
  Then, the only possibility is $a=1,b=c=0$ and
  $$
 \deg(C)=12,\quad p_a(C)=12,\quad g(C)=0.
  $$ In this case, we have $C\cdot B=36$, where $B=R\cap R'$ is an irreducible curve of degree 18. On the other hand, since both $C$ and $B$ are singular at a common orbit  of length 12, we have $C\cdot B\ge12\cdot 4=48$. We obtain a contradiction. So $C$ is smooth.
        \item We know that $C$ is smooth, and we have the bound \eqref{eqn:boundgenus} on its genus. 
        First, when $\deg(C)=8$, we find all such curves in Lemma~\ref{lemm:allcurve80} and it follows that $g(C)=0$.
        %But $g(C)=4$ can be excluded as follows. By Riemann-Roch theorem, we have $h^0(2H-C)\ge2+\frac{(2H-C)^2}{2}=1$. Hence, the linear system $|2H-C|$ is non-empty, and since $(2H-C)^2=-2$, it has a fixed curve $F$ of degree at most $\deg(2H-C)=4$. The only $G$-invariant curves of degree at most four on $X$ are $C_4$ and $C_4'$. However, those curves do not belong to $S$ since $S\ne R, R'$, so we get a contradiction.
        When $\deg(C)=10$, recall that $C$ contains no orbit of length 12. Then by \cite[
Lemma 5.1.5]{CheltsovShramov}, or by searching through the database of curves with $\fA_5$-actions in \cite{lmfdb}, we find that $g(C)=6$.  Finally, if $\deg(C)=12$, similarly going through the classification, we get $g(C)=0$.
    \end{enumerate}
\end{proof}

First, we exclude several curves as possible non-canonical centers.

\begin{lemm}\label{lemm:56conics}
    Let $C$ be a $G$-invariant union of $r$ conics not contained in $R$ or $R'$, with $r\in\{5,6\}$. Then each irreducible component of  $C$ is not a center of non-canonical singularities of $(X,\lambda\mls_X)$.
\end{lemm}

%{\bf Here is another proof according to Vanya:}
    \begin{proof}
    By Lemma \ref{lemma:curves not in non normal K3}, there exists a unique smooth K3 surface $S$ in the pencil $\mathcal P$ containing $C$. Let $H_S$ be a general hyperplane section of $S$ and $m=\mult_C(\lambda\mls_X)$. Assume that irreducible components of $C$ are non-canonical centers. Then $m>1$, and we have $$
\lambda\mls_X\vert_S\sim_{\dq}mC+\Delta,\quad m\ge\mult_C(\lambda\mathcal M_X)>1
    $$ 
    for some effective divisor $\Delta$ on $S$ not supported along $C$. It follows that the divisor 
    $$
    3H_S-C\sim_\dq\Delta+(m-1)C
    $$
    is effective. On the other hand, consider the divisor on $S$ given by 
    $$
    D=
        (r-1)H_S-C.
    $$
    The equations of $C$ are given in Lemma~\ref{lemm: G-inv red curves}. By computation, we check that  the linear subsystem in $|\cO_X(r-1)|$ consisting of surfaces passing through $C$ does not contain any base curve other than $C$. It follows that $D$ is nef. However, we compute 
    $$D\cdot(3H_S-C)=\begin{cases}
        -8&\text{if }r=5,\\
        -18&\text{if } r=6,
    \end{cases}$$ 
    which gives a contradiction.
\end{proof}

\begin{lemm}\label{lemm:deg12}
    Let $C$ be a $G$-invariant curve of degree 12 not contained in $R$ or $R'$. Then each irreducible component of $C$ is not a center of non-canonical singularities of $(X,\lambda\mls_X)$.
\end{lemm}
\begin{proof}
    If $C$ is reducible, it is a union of 6 conics by Lemma~\ref{lemm: G-inv red curves}. The assertion follows from Lemma~\ref{lemm:56conics}. We assume that $C$ is irreducible.

    By Lemma \ref{lemma:curves not in non normal K3}, the curve $C$ is rational, and there exists a smooth K3 surface $S$ in the pencil $\mathcal P$ containing $C.$ Let $H_S$ be a general hyperplane section on $S$. We have $(4H_S-C)^2=-2$, and Riemann-Roch theorem gives $h^0(4H_S-C)\ge1$. Assume that it is an equality. Then the only element of $|4H_S-C|$ is the class of a $G$-invariant curve $C'$ of degree 12.  If $C'$ is reducible, by Lemma~\ref{lemm: G-inv red curves}, $C'$ is either a union of 12 lines or 6 conics. None of these is possible: the 12 lines are contained in $R$ or $R'$; and $C'^2=-12$ if $C'$ is a union of 6 conics. It follows that $C'$ is irreducible, in which case we can exclude $C'$ as a non-canonical center the same way as in the proof of Lemma~\ref{lemma:irrcurves not in Q}.
    
    Assume now that $h^0(4H_S-C)>1$. We show that this is impossible. The linear system $|4H_S-C|$ splits into a fixed part $\mathcal F$ and a mobile part $\mathcal G$. But $(4H_S-C)^2=-2$, so $|4H_S-C|$ is not nef. We deduce that $\mathcal F$ is not empty. Let $F\in\cF$. The degree of $F$ is 8 or 10. The latter case would imply that curves in $\cG$ have degree 2, i.e., there is a pencial of rational curves in $S$, which is impossible on K3 surfaces. Assume that $\deg(F)=8$. Then the degree of a general member $M\in\mathcal G$ is four. Either $M$ is a smooth elliptic curve or $M$ is rational.
    Again, the latter is impossible on K3 surfaces. So $M$ is a smooth elliptic curve. Consider the matrix $$A=\begin{pmatrix}H_S^2&H_S\cdot F&H_S\cdot M\\F\cdot H_S&F^2&F\cdot M\\M\cdot H_S&M\cdot F&M^2\end{pmatrix}.$$ Since $\pic^G(S)$ is of rank at most two by \cite[Proposition 6.7.3]{CS}, the determinant of $A$ must be zero. But we get $$\det(A)=\det\begin{pmatrix}6&8&4\\8&-2&7\\4&7&0\end{pmatrix}=186,
    $$
    hence we obtain a contradiction.
\end{proof}

We turn to $G$-invariant curves of degree 16, which are necessarily irreducible. The strategy of the proof is similar, but such curves
may be singular. We first treat smooth curves here. The case of singular curves will be excluded at the end of this subsection, where we explicitly find all such curves in equations.
\begin{lemm}\label{lemm:deg16smooth}
    Let $C$ be a smooth irreducible $G$-invariant curve of degree 16 in $X$ that is not contained in $R\cup R^\prime$. Then $C$ is not a non-canonical center of $(X,\lambda\cM_X)$.
\end{lemm}
\begin{proof}
    Let $S$ be the unique smooth K3 surface in $\cP$ containing $C$, and $H_S$ a general hyperplane section on $S$. Similarly as in Lemma~\ref{lemm:deg12}, we know that $3H_S-C$ is effective and we seek for a contradiction by finding a nef divisor on $S$ intersecting $3H_S-C$ negatively.
    
     The Hodge index theorem implies that 
    $$
    C^2\leq\frac{(C\cdot H_S)^2}{(H_S)^2}=\frac{128}{3}\quad\Longrightarrow \quad C^2\leq 42, \quad g\leq 22,
    $$ 
    where $g=g(C)$ is the genus of $C$. Possible genera of smooth irreducible curves with $\fA_5$-actions and their orbit structures are classified in \cite[Lemma 5.1.5]{CheltsovShramov}. Recall from the proof of Lemma~\ref{lemm: G-inv red curves} that $C$ contains at least one $G$-orbit of length 12. By \cite[Lemma 5.1.5]{CheltsovShramov}, we find that 
    $$
    g\in\{0,4,5,9,10,13,15,19,20\},
    $$
so $C^2=2g-2\in\{-2,6,8,16,18,24,28,36,38\}$. Put
    $$
    n=\begin{cases}
        4&\text{ if } g\in\{19, 20\},\\
        5&\text{ if } g\in\{9,10,13,15\},\\
        6&\text{ if } g\in\{0,4,5\}.\\
    \end{cases}
    $$
One can check that $(nH_S-C)^2\geq 0$, and 
$$
(nH_S-C)\cdot(3H_S-C)=2n-48+C^2<0. 
$$
Therefore, if $nH_S-C$ is nef, we are done. Now let us show that $nH_S-C$ is nef for all possible genera. 

Assume that $nH_S-C$ is not nef. By Riemann-Roch, $(nH_S-C)^2\geq 0$ implies that $h^0(nH_S-C)\geq 2$. So $|nH_S-C|$ has a mobile part.
Moreover, since $nH_S-C$ is not nef, $|nH_S-C|$ has a $G$-irreducible fixed component $F$ such that $\deg(F)\leq \deg(nH_S-C)$. 
The curves in the mobile part cannot be rational because $S$ is a K3 surface, and thus their degree is at least 4. It follows that 
$$
\deg(F)\leq\deg(nH_S-C)-4=
    \begin{cases}
        4&\text{ if } g\in\{19, 20\},\\
        10&\text{ if } g\in\{9,10,13,15\},\\
        16&\text{ if } g\in\{0,4,5\}.\\
    \end{cases}
$$
Lemma~\ref{lemma:curves not in non normal K3} also implies that $\deg(F)\in\{8,10,12,16\}$. This immediately excludes the possibility $g\in\{19,20\}$. In the other two cases, since $\mathrm{rk}\,\Pic^G(S)\leqslant 2$, the intersection matrix of $F,H_S$ and $C$ is degenerate. In particular, let $x=F\cdot C$, we have
\begin{align}\label{eqn:matdegen}
    \mathrm{det}\begin{pmatrix}
    F^2&x&\deg(F)\\
    x&C^2&16\\
    \deg(F)&16&6
\end{pmatrix}=0,
\end{align}
which gives a quadratic equation in $x$. We show that this equation does not have integer solutions satisfying the geometric conditions. 
Since $F$ is a fixed component, we know that $h^0(F)=1$, and by Riemann-Roch, $F^2<0$. Hence, if $F$ is irreducible, then, by adjunction formula, we have $F^2=-2$. When $F$ is reducible, $F^2$ is supplied by Lemma~\ref{lemm: G-inv red curves}. In particular, we have
$$
F^2=\begin{cases}
    -12&\text{ if } \deg(F)=12 \text{ and $F$ is reducible},\\
     -10&\text{ if } \deg(F)=10 \text{ and $F$ is reducible},\\
      -2&\text{ if } \text{ $F$ is irreducible},\\
\end{cases}
$$
Now, running through all possibilities, we find that \eqref{eqn:matdegen} has an integer solution only in the following two cases:
\begin{enumerate}
    \item $n=6,\,\,  \deg(F)=16,\,\, F^2=-2, \,\,C^2=-2,\,\, C\cdot F=-2$,
    \item$n=5,\,\,  \deg(F)=10,\,\, F^2=-10, \,\,C^2=16,\,\, C\cdot F=0$.
\end{enumerate}
So, in the first case, we have $C=F$, and in the second case, $C$ is a union of 5 conics.
In both cases, we know that the linear system  $|nH-C-F|$ is not empty since it has the same mobile part as $|nH-C|$. 
One can compute that  $(nH-C-F)^2<0$, implying that $|nH-C-F|$ has a fixed component of degree 4, which is impossible. 
We obtain a contradiction, and this completes the proof. 
\end{proof}

Now we turn to irreducible curves of degree 8 or 10. Such curves can indeed be non-canonical centers. We characterize the Sarkisov links arising from them. The following result will allow us to prove in Section~\ref{sect:theproof} that up to some $G$-birational self-map of $X$ which normalizes the image of $G$ in $\Aut(X)$, the pair $(X,\lambda\cM_X)$ is canonical away from $\Sigma_5\cup\Sigma_5'$.

\begin{lemm}\label{lemm:linkC8C10}
Let $Z$ be a $G$-invariant smooth irreducible curve $Z$ not contained in $R\cup R'$, of degree 8 and genus 0, or of degree 10 and genus 6. Assume that $Z$ is a non-canonical center of $(X,\lambda\mls_X)$. Let $\varphi:\widetilde X\to X$ be the blowup of $Z$. Then  $-K_{\widetilde X}$ is big and nef.
Moreover, for $n\gg 0$, the linear system $|n(-K_{\widetilde{X}})|$ is base point free and gives a small birational map $\psi: \widetilde X\to V$. There exists the following $G$-equivariant commutative diagram
\begin{equation}\label{diag}
    \xymatrix{
& \widetilde{X} \ar@{-->}[rr]^{\chi} \ar[dl]_{\varphi} \ar[dr]^{\psi} & & \widetilde{X}' \ar[dl]_{\psi'} \ar[dr]^{\varphi'}& \\
X \ar@{-->}@/_1pc/[rrrr]_{\delta} & & V& & X'
}
\end{equation}
where
\begin{enumerate}
    \item $\chi$ is a composition of flops,
    \item  $\psi'$ is also a small birational map,
    \item $\varphi'$ is a $K_{\widetilde X'}$-negative extremal contraction, 
      \item $X'$ is also a smooth quadric threefold, and $\varphi'$ is the blowup of a curve $Z'\subset X'$ of the same degree and genus as $Z$,
   \item $X$ and $X'$ are $G$-isomorphic, i.e. the birational map $\delta$ normalizes the image of $G$ in $\Aut(X)$.
     \end{enumerate}
\end{lemm}
\begin{proof}

 First, we introduce some notation. Let $g(Z)$ be the genus of $Z$, $E$ the exceptional divisor of $\varphi$, $H$ a general hyperplane section on $X$, and $\widetilde H$ the pullback of $H$ to $\widetilde X$. Let $S$ be the unique K3 surface in $\cP$ containing $Z$, $\widetilde S$ its strict transform on $\widetilde X$, and $H_S$ the restriction of $H$ to $S$. Note that $\widetilde S\simeq S$ since $S$ is smooth. 
 
To show that $-K_{\widetilde X}$ is big, we compute $$
(-K_{\widetilde X})^3=(3\widetilde H-E)^3=2g(Z)-6\cdot\deg(C)+52=4>0.
$$

To show that $-K_{\widetilde X}$ is nef, it suffices to show that $|3H_S-Z|$ contains no other base curve than $Z$, i.e., $3H_S-Z$ is nef. Indeed, if $-K_{\widetilde X}$ is not nef, then the divisor  
$
-K_{\widetilde X}\vert_{\widetilde S}=(3\widetilde H-E)\vert_{\widetilde S}=3H_S-Z
$
is also not nef.

    Let us first do this when $Z$ is a rational curve of degree 8. By Riemann-Roch, $h^0(3H_S-Z)=2+\frac{(3H_S-Z)^2}{2}=4$, so the degree 10 linear system $|3H_S-Z|$ has a non-trivial mobile part. Hence a fixed curve $F$ of this linear system has degree at most 9. Lemma \ref{lemma:curves not in non normal K3} implies that $F$ is of degree 8. But then the mobile part of $3H_S-Z$ contains a pencil of rational curves, which is impossible on K3 surfaces. We conclude that $(3H_S-Z)$ does not contain a base curve other than $Z$.
    
    Similarly, if $Z$ has degree 10 and genus 6.  Riemann-Roch theorem implies that $h^0(3H_S-Z)=2+\frac{(3H_S-Z)^2}{2}=4$, so the degree 8 linear system $|3H_S-Z|$ has a nontrivial mobile part. Hence, a fixed curve $F$ of this linear system is of degree at most 7.  Lemma \ref{lemma:curves not in non normal K3} implies that this is impossible. 

We prove that $-K_{\widetilde X}$ is big and nef. Then it follows from base point free theorem that 
the linear system $|n(-K_{\widetilde{X}})|$ is base point free for $n\gg 0$, and it gives a birational map $\psi: \widetilde X\to V$.
Either $\psi$  contracts a divisor, or $\psi$ is small. 
However, the former case is impossible, because we assume that $Z$ is a center of non-canonical singularities of $(X,\lambda\mls_X)$. If $\psi$ contracts a divisor, then this divisor must be a fixed component of the linear system $\mls_X$, which is impossible, since $\mls_X$ is mobile by assumption.

Hence, we see that $\psi$ is a small birational contraction. Then the existence of $G$-equivariant commutative diagram and (1) -- (3) follow from the Sarkisov program. This is a type II Sarkisov link. Moreover, (4) follows from matching the numerical invariants with a classification of such links in \cite{JahnkePeternellRadloff2005,JahnkePeternellRadloff2011,cutrone2025update}. Our cases correspond to rows 71 and 72 of Table 1 in \cite{cutrone2025update}. In particular, we find $X'$ is also a smooth quadric.  To show (5), we notice that any other action of $\ac$ on the quadric $X'$ has an invariant hyperplane section, which would intersect $Z'$ in $\deg(Z')=8$ or $10$ points. But, by \cite{CS}, the smallest possible orbit of $\ac$ on a smooth curve is of length 12, so we get a contradiction.
\end{proof}
\begin{rema}
   Lemma~\ref{lemm:linkC8C10} shows that if such a Sarkisov link exists, it necessarily leads to a $G$-isomorphic quadric and gives no new $G$-Mori fibre space.  Remark~\ref{rema:c8curve} and Lemma~\ref{lemm:allcurve80} finds all curves of degree 8 satisfying the assumptions of Lemma~\ref{lemm:linkC8C10}. On the other hand, we do not know the existence of such a curve of degree 10. 
\end{rema}

\begin{comment}
    \begin{rema}
    Links satisfying assumptions in Lemma~\ref{lemm:link} have been classified in \cite{CutroneMarshburn2013}. All possible cases are No.71 and 72 in \cite[Table 3]{CutroneMarshburn2013}. We find that $Z$ is necessarily also a smooth quadric threefold and $\varphi'$ is the blowup of a curve in $Z$ with the same degree and genus as $C$.
\end{rema}
\end{comment}

\begin{comment}
\begin{lemm}\label{lemma:deg8curve}
    The only $G$-invariant curves $C_8$ of degree eight are \color{red}something\color{black}. The only surfaces containing $C_8$ in the pencil $\mathcal P$ are $S_{a_1,a_2}$ where
    $$
    a_1=41,\quad a_2=(-63\zeta_5^3 - 63\zeta_5^2 + 50)
    $$
    or 
     $$
    a_1=41,\quad a_2=(63\zeta_5^3 + 63\zeta_5^2 + 113)
    $$
\end{lemm}
\begin{proof}
    \color{red}To be done.\color{black}
\end{proof}
\end{comment}

Finally, we show that the only curves in $R$ or $R'$ which can be non-canonical centers are $C_4$ and $C_4'$. The Sarkisov links centered at them have been studied in \cite[Section 5.9]{Chelcalabi} and \cite{malbon2025kstablefanothreefoldsrank}.

\begin{lemm}[{\cite[Lemma 7]{malbon2025kstablefanothreefoldsrank}}]\label{lemm:linkC4}
Let $H$ be a general hyperplane section on $X$. Then the linear system $|2H-C_4|$ gives rise to a $G$-equivariantly birational involution $\varphi: X\dashrightarrow X$. There exists 
a $G$-equivariant commutative diagram
   $$
     \xymatrix{
\widetilde X\ar@{->}[rr]^{\widetilde\varphi}\ar@{->}[d]_{\pi}&&\widetilde X\ar@{->}[d]^{\pi}\\
\ X\ar@{-->}[rr]_{\varphi}&&\ X
}
$$
where  $\pi:\widetilde X\to X$ is the blowup of $C_4$, $\widetilde{R}$ is the strict transform on $\widetilde{X}$ of the surface $R$, and $\widetilde \varphi\in\Aut(\widetilde X)$ has order 2. 
Let $E$ be the exceptional divisor of $\pi$. Then $\widetilde{\varphi}(E)=\widetilde{R}$.
Moreover, one has $\widetilde R\simeq E\simeq\bP^1\times\bP^1$.
\end{lemm}

Note that $\widetilde{X}$ above is a smooth Fano threefold of Picard rank 2 and degree 28. More details about $\widetilde X$ can be found in \cite{Joeauto} or \cite[Section 5.9]{Chelcalabi}.

\begin{rema}\label{rema:c8curve}
Recall that there is an involution $\sigma\in\Aut(X)$ such that $\langle \sigma, G\rangle\simeq\fS_5$. In many cases, $G$-orbits or $G$-invariant curves appear in pairs swapped by $\sigma$. For example, $\sigma$ swaps $C_4$ and $C_4'$. By symmetry, $|2H- C_4'|$ gives an involution $\varphi'$ similar to $\varphi$. We construct $\varphi$ and $\varphi'$ in equations and find that $\varphi(C_4')$ and $\varphi'(C_4)$ are two smooth irreducible curves of degree 8. Each curve is cut out by cubics passing through it. We also find that $\varphi'(\varphi(C_4'))$ and $\varphi(\varphi'(C_4))$ are irreducible curves of degree 16 such that each of them has 12 cusps. Equations of $\varphi, \varphi'$ and the curves can be found in \cite{Z-web}.
\end{rema}

\begin{lemm}\label{lemm:quad2curveinR}
    Let $C$ be a $G$-invariant curve of degree at most 17 contained in $R$ or $R'$ and $C\ne C_4,C_4'$. Then each irreducible component of $C$ is not a center of non-canonical singularities of $(X,\lambda\mls_X)$.
\end{lemm}

\begin{proof}
Without loss of generality, we may assume that $C\subset R$. Suppose that $\mathrm{mult}_C(\lambda\mls_X)>1$. Let us seek for a contradiction. We use the notation of Lemma~\ref{lemm:linkC4}. Set $\widetilde{H}=\pi^*(H)$, let $\mls_{\widetilde{X}}$ be the strict transform on $\widetilde{X}$ of the linear system $\mls_X$, and $\widetilde{C}$ the strict transform on $\widetilde{X}$ of the curve $C$. Then 
$$
\widetilde R\sim 2\widetilde{H}-3E, \quad
\widetilde{H}\cdot \widetilde{C}\leq 17,\quad \widetilde{C}\not\subset E,\quad\text{and}\quad 
\mathrm{mult}_{\widetilde{C}}(\lambda\mls_{\widetilde{X}})>1,
$$
where 
$$
\lambda\mls_{\widetilde{X}}\sim_{\mathbb{Q}} 3\widetilde{H}-rE
$$
for some $r\in\mathbb{Q}_{\geqslant 0}$. By Lemma~\ref{lemm:linkC4}, there is an involution $\widetilde{\varphi}\in\Aut(\widetilde{X})$ such that $\widetilde{\varphi}(\widetilde{R})=E$ and 
$$
\widetilde{\varphi}^*(\widetilde{H})\sim 2\widetilde{H}-E.
$$
In particular, we know that $\widetilde{R}\simeq\mathbb{P}^1\times\mathbb{P}^1$.
Moreover, we have $\widetilde{R}\vert_{E}=2\Delta$, where $\Delta$ is a smooth curve in $E$ of bidegree $(1,1)$.
This implies that the $G$-action on $E$ is diagonal.

To obtain a contradiction, we consider the restriction 
$$
\lambda\mls_{\widetilde{X}}\vert_{\widetilde{R}}\sim_{\mathbb{Q}} (3\widetilde{H}-rE)\vert_{\widetilde{R}}
$$
and show that the inequality $\mathrm{mult}_{\widetilde{C}}(\lambda\mls_{\widetilde{X}}\vert_{\widetilde{R}})>1$ contradicts 
$$
\widetilde{H}\cdot \widetilde{C}\leq 17,
$$ since $\widetilde{C}\ne\Delta$. To do this, we restrict $\widetilde{\varphi}(\mls_{\widetilde{X}})$ to $E$. 

Namely, set $\mls_{\widetilde{X}}^\prime=\widetilde{\varphi}(\mls_{\widetilde{X}})$ and $\widetilde{C}^\prime=\widetilde{\varphi}(\widetilde{C})$. Then   
$(2\widetilde{H}-E)\cdot \widetilde{C}'\leq17$, $\widetilde{C}'\ne\Delta$ and
$$
\mathrm{mult}_{\widetilde{C}^\prime}(\lambda\mls_{\widetilde{X}}^\prime)>1,
$$
where 
$$
\lambda\mls_{\widetilde{X}}^\prime\sim_{\mathbb{Q}} 3(2\widetilde{H}-E)-r\widetilde{R}.
$$
Let $f$ be a fibre of the natural projection $\pi\vert_E:E\to C_4$, and $s$ a section of this projection such that $s^2=0$. Then 
$$
\widetilde{C}^\prime\sim as+bf
$$
for some non-negative integers $a$ and $b$.
We have 
\begin{equation}\label{eqn:degC17}
(as+bf)\cdot (2\widetilde{H}-E)|_E=(2\widetilde{H}-E)\cdot \widetilde{C}^\prime\leq 17.
\end{equation}
We compute $2\widetilde{H}\vert_E\sim 8f$ and $E\vert_E=s-5f$, so $(2\widetilde{H}-E)\vert_E\sim s+3f$.
Plugging this into \eqref{eqn:degC17}, we get 
$$
3a+b\leq 17.
$$
Since $\widetilde C'\ne \Delta$, we know that $a\ne b$. A computation of $G$-invariant forms on $E$ then implies that 
\begin{align}\label{eqn:ab}
    (a,b)\in\{(0,12),(1,11),(1,13),(2,10)\}.
\end{align}
Now, we use the inequality $m:=\mathrm{mult}_{\widetilde{C}^\prime}(\lambda\mls_{\widetilde{X}}^\prime)>1$. It gives 
$$
\lambda\mls_{\widetilde{X}}^\prime\vert_{E}=m\widetilde{C}^\prime+\Omega,
$$
where $\Omega$ is a $\mathbb{Q}$-linear system on $E$.
On the other hand, we have 
$$
\lambda\mls_{\widetilde{X}}^\prime\vert_{E}\sim_{\mathbb{Q}} 3(s+3f)-2r(s+f)=(3-2r)s+(9-2r)f,
$$ 
and thus $m(as+bf)+\Omega\sim_{\mathbb{Q}}  (3-2r)s+(9-2r)f$. This yields
$$
b<bm\leqslant s\cdot (m(as+bf)+\Omega)=s\cdot((3-2r)s+(9-2r)f)=9-2r\leqslant 9,
$$
which contradicts \eqref{eqn:ab}. This completes the proof.
\end{proof}

Using the geometry of $\widetilde X$, we can show that $C_8$ and $C_8'$ described in Remark~\ref{rema:c8curve} are the only $G$-invariant rational curves of degree 8 in $X$.

\begin{lemm}\label{lemm:allcurve80}
    Let $C$ be a $G$-invariant curve of degree 4 in $X$. Then $C=C_4$ or $C_4'$. Let $C$ be a $G$-invariant rational curve of degree 8 in $X$. Then $C=C_8$ or $C_8'$, where $C_8=\varphi(C_4')$ and $C_8'=\varphi'(C_4)$.
\end{lemm}
\begin{proof}
By Lemma~\ref{lemma:curves not in non normal K3}/(1), any $G$-invariant curve $C$ of degree 4 is contained in $R$ or $R'$. Without loss of generalities, assume $C\subset R$. The first assertion then follows from the proof of Lemma~\ref{lemm:quad2curveinR}.  In particular, no solution in \eqref{eqn:ab} gives $3a+b=4$.
 To show the second assertion, recall that any $G$-invariant curve of degree 8 is irreducible since 8 is not a multiple of the index of any strict subgroup of $G$. By the proof of Lemma~\ref{lemma:curves not in non normal K3}/(1), we know that $C$ contains an orbit of length 12.  Assume that $\Sigma_{12}\subset C$. Under the same notation as in the proof of Lemma~\ref{lemm:quad2curveinR}, let $\widetilde C$ be the strict transform of $C$ in $\widetilde X$. We have 
 $$
 \widetilde C\cdot (2\widetilde H-E)=2\cdot 8-12=4.
 $$
 It follows that $\varphi(C)$ is a curve of degree 4 in $R'$, which is necessarily $C_4'$. Then $C=\varphi(C_4')$ since $\varphi$ is an involution. Similarly, we can show that $C=\varphi'(C_4)$ when $\Sigma_{12}'\in C$,
\end{proof}

\begin{lemm}\label{lemm:deg16sing}
    Let $C$ be a singular $G$-invariant curve of degree 16 that is not contained in $R\cup R^\prime$. Then $C=\varphi'(C_8)$ or $C=\varphi(C_8')$. Moreover, $C$ is not a non-canonical center of the log pair $(X,\lambda\cM_X)$. 
\end{lemm}
\begin{proof}
First, we show that $C$ is singular along an orbit of length 12. Assume it is not. Then $|\Sing(C)|\geq 20$. Let $S$ be the unique smooth K3 surface in $\cP$ containing $C$ as in Lemma~\ref{lemma:curves not in non normal K3}, and $H_S$ a general hyperplane section on $S$. We have $C^2\leq 42$ by the Hodge index theorem, and thus the arithmetic genus $p_a(C)\leq 22$. Since $|\Sing(C)|\geq 20$, the geometric genus $g(C)\leq 2$. By \cite[Lemma 5.1.5.]{CheltsovShramov}, $\fA_5$ does not act on curves of geometric genus 1 or 2. It follows that 
$$
g(C)=0,\quad |\Sing(C)|=p_a(C)=20, \quad C^2=38.
$$
Then $(C-2H_S)^2=-2$. This implies that $|C-2H_S|$ is not empty and has a fixed component of degree $\leq 4$, which is impossible by Lemma~\ref{lemma:curves not in non normal K3}. Thus, $C$ is singular along an orbit of length 12. Assume that $C$ is singular at points in $\Sigma_{12}$. Under the same notation as in the proof of Lemma~\ref{lemm:allcurve80}, we have that 
$$
\deg(\varphi(C))=\widetilde C \cdot (2\widetilde H-E)=2\cdot 16-2\cdot 12-(a\cdot12+b\cdot20+c\cdot 30+d\cdot60)\geq 0
$$
for $a,b,c,d\in\bZ_{\geq0}$. Then the only possibility is $\deg(\varphi(C))=8$. It follows from Lemma~\ref{lemm:allcurve80} that 
$
\varphi(C)=C_8 
$ or $C_8'$, i.e., $C=\varphi(C_8)$ or $\varphi(C_8')$. The first is impossible since $\varphi(C_8)=C_4'$ has degree 4. Thus, $C=\varphi(C_8')$. Similarly, if $C$ is singular at points in $\Sigma_{12}'$, we can show that $C=\varphi'(C_8)$.

Then we can find equations of $C$. We listed them in \cite{Z-web}. Using equations, we check that $|\Sing(C)|=12$ and $|5H_S-C|$ contains no base curve other than $C$. Then $5H_S-C$ is nef. Since $C_8$ and $C_8'$ are rational curves, we know $g(C)=0$. It follows that $p_a(C)=12$ and $C^2=22$.

Now assume that $C$ is a non-canonical center of $(X,\lambda\cM_X)$. Similarly as in the proof of Lemma~\ref{lemm:deg16smooth}, we know that $3H_S-C$ is effective. But computing 
$$
(5H_S-C)\cdot(3H_S-C)=-16<0,
$$
we obtain a contradiction to the nefness of $5H_S-C$. This completes the proof. 

\end{proof}
\subsection{Orbits of points}

Now we study when all non-canonical centers are points.
First, we show that points in the invariant curves of degree 4 or 8 cannot be non-canonical centers in this case.
\begin{lemm}\label{lemm:ptsinC4}
    Suppose that $C_4$ is not a non-canonical center of the log pair $(X,\lambda\mls_X)$, then every point in $C_4$ is not a non-canonical center. The same holds for $C_4'$.
\end{lemm}
\begin{proof}
    Let $a=\mult_{C_4}(\lambda \cM_X)$. By assumption, we have $a\leq 1.$ We retain the notation in Lemma~\ref{lemm:quad2curveinR}: let $\pi:\widetilde X\to X$ be the blowup of $C_4$, and $E$ its exceptional divisor, so that $E=\bP^1\times\bP^1\to C_4$. Let $\lambda\mls_{\widetilde{X}}$ be the linear system satisfying
  $$
   K_{\widetilde  X}+\lambda\mls_{\widetilde{X}}+(a-1)E\sim \varphi^*(K_X+\lambda M_X).
  $$
  Assume that a point on $C_4$ is a non-canonical center of $(X,\lambda\cM_X)$. Then there exists a center $Z$ of non-canonical singularities of the pair $$
  (\widetilde X,\lambda\mls_{\widetilde{X}}+(a-1)E)
  $$ such that $Z\subset E$. It follows that $Z$ is a center of non-canonical singularities of  $(\widetilde X,\lambda\mls_{\widetilde{X}}+E)$. By inversion of adjunction, $Z$ is a non-log-canonical center of $(E,\lambda\cM_{\widetilde X}\vert_E)$. Note that $\lambda\cM_{\widetilde X}\vert_E$ is not mobile, and consider a divisor $D\sim_\dq\lambda\cM_{\widetilde X}\vert_E$.
   Let $f$ be a general fibre of $E\to C_4$ and $s$ a section such that $s^2=0$. We compute 
   $$
   D\sim_\dq as+(12-5a)f\in\Pic(E)\otimes\bQ.
   $$
   Since $a\leq 1$, we know that $(E,D)$ is log-canonical at a general point of any curve which is not a fibre. On the other hand, if any fibre is a non-log-canonical center of $(E,D)$, then at least 12 fibres are non-log-canonical centers, since the smallest orbit of the $\fA_5$-action on $C=\bP^1$ has length 12. This is impossible because $12-5a\leq12$. It follows that $(E,D)$ is not log-canonical at finitely many points. Let $p$ be one of these points and $L$ a fibre containing $p$. Write 
   $$
   D\sim_\dq bL+\Delta,\quad  b\leq 1
   $$
   for some divisor $\Delta$ not supported along $L$. Then $(E,L+\Delta)$ is also not log-canonical at $p$. By inversion of adjunction, $(L,\Delta\vert_L)$ is not log-canonical at $p$, which contradicts
   $$
   (\Delta\cdot L)_p=a\leq 1.
   $$
\end{proof}

\begin{lemm}\label{lemm:ptsinC8}
    Suppose that the curves $C_4,C_4',C_8$ and $C_8'$ are not centers of non-canonical singularities of $(X,\lambda\mls_X)$, then none of the points in $C_4\cup C_4'\cup C_8\cup C_8'$ is a non-canonical center.
\end{lemm}
\begin{proof}
 By Lemma~\ref{lemm:ptsinC4}, it suffices to show that $(X,\lambda \cM_X)$ is canonical at any point in $C_8$ under the assumption.
   Let $\Sigma$ be the intersection of $C_8$ with the non-log-canonical locus of $(X,\lambda\cM_X)$. Similarly as before, we may assume that $\Sigma$
  is 0-dimensional. Let $M_1,M_2$ be two general members in $\cM_X$, and write 
  $$
  \lambda^2(M_1\cdot M_2)=mC_8+\Delta
  $$
for some divisor $\Delta$ not supported along $C_8$ and $m\geq 0$.  Intersecting with a general hyperplane $H$, we obtain 
$$
 18=\lambda^2(M_1\cdot M_2\cdot H)\geq m\deg(C_8)=8m\quad \Rightarrow\quad m\leq9/4.
$$
Recall that $C_8$ is cut out by cubics, see Remark~\ref{rema:c8curve}. Let $S$ be a general cubic on $X$ passing through $C_8$. Since $C_8$ is smooth, by Theorem~\ref{theo:4n2}, we have 
$$
\mult_{\Sigma}\left(\lambda^2(M_1\cdot M_2)\right)>4\quad\Rightarrow\quad\mult_{\Sigma}(\Delta)\geq 4-m.
$$
Observe that
\begin{align}\label{eqn:ineq}
   54=\lambda^2(M_1\cdot M_2\cdot S)=24m+\Delta\cdot S\geq 24m+|\Sigma|(4-m).
\end{align}
Since $4-m>0$, the inequality \eqref{eqn:ineq} implies that $|\Sigma|<20$. By Lemma~\ref{lemm:2a5orbit}, $\Sigma$ consists of orbits of length 5 or 12. Orbits of length 12 are in $C_4\cup C_4'$, and thus are excluded by Lemma~\ref{lemm:ptsinC4}. On the other hand, none of orbits of length 5 is in $C_8$. It follows that $\Sigma=\emptyset.$
\end{proof}

\begin{prop}\label{prop:ptcase2}
    Suppose that the log pair $(X,\lambda\mls_X)$ is canonical away from finitely many points. Then it is canonical away from $\Sigma_5\cup\Sigma_5'$.
\end{prop}
\begin{proof}
Let $\Sigma$ be the non-canonical locus of $(X,\lambda\cM_X).$ By Remark~\ref{remark:Ziquan}, $(X,\frac32 \lambda\cM_X)$ is not log-canonical at points in $\Sigma$. Let $\varepsilon$ be a positive rational number such that
    $$
    \Sigma\subset\Omega,\quad \Omega:=\mathrm{Nklt}(X,(\frac{3}{2}-\varepsilon)\lambda\cM_X).
    $$
Assume that $\Omega$ contains some curve $C$. Let 
$$
m=\mult_C((\frac{3}{2}-\varepsilon)\lambda\cM_X).
$$
Observe that
$$
1<m<\frac32\quad\Rightarrow \quad \frac{m^2}{m-1}>\frac92.
$$
Let $M_1, M_2\in\cM_X$ be two general elements, and $H$ a general hyperplane on $X$. By Theorem~\ref{theo:demailly}, we know that 
$$
\frac92\deg(C)<\frac{m^2}{m-1}\deg(C)\leq (\frac32-\varepsilon)^2\lambda^2(H\cdot M_1\cdot M_2)<\frac9418=\frac{81}{2},
$$
which implies that $\deg(C)\leq 8.$ It follows that $C$ can only be one of $C_4, C_4', C_8$ and $C_8'$. By Lemma~\ref{lemm:ptsinC4} and Lemma \ref{lemm:ptsinC8}, $\Sigma$ is disjoint from these curves. Thus, the 0-dimensional component $\Omega_0$ of $\Omega$ is non-empty. Applying Nadel vanishing in the same way as in the proof of Proposition~\ref{prop:pointsnotinQ}, we obtain that $|\Sigma|\leq|\Omega_0|<14$. Since all orbits of length 12 are contained in $C_4$ or $C_4'$, the proof is complete.
\end{proof}

\section{The nonstandard $\fA_5$-action on the cubic threefold}\label{sect:a5cub2}

Now, we focus on the other model of $X$: a cubic threefold $Y$ with $5\sA_2$-singularities, carrying the same $G$-action generated by \eqref{eqn:2a5gen}. Let $f_1$ and $f_2$ be the cubics defined in \eqref{eqn:2a5cubicpencil}.
%\begin{multline*}
  % f_1=x_1^2x_2 + x_1x_2^2 + 2x_1x_2x_3 + x_2^2x_3 + x_2x_3^2 + 2x_2x_3x_4 + 
    %        x_3^2x_4 + x_3x_4^2 +\\+x_1^2x_5 + 2x_1x_2x_5 + 2x_1x_4x_5 + 2x_3x_4x_5 +
    %        x_4^2x_5 + x_1x_5^2 + x_4x_5^2,
%\end{multline*}
%\begin{multline*}
  %  f_2=x_1^2x_3 + x_1x_3^2 + x_1^2x_4 + 2x_1x_2x_4 + x_2^2x_4 + 2x_1x_3x_4 + 
  %          x_1x_4^2 + x_2x_4^2 + \\+x_2^2x_5 + 2x_1x_3x_5 + 2x_2x_3x_5 + x_3^2x_5 + 
  %          2x_2x_4x_5 + x_2x_5^2 + x_3x_5^2.
%\end{multline*}
Then $Y$ is given by 
\begin{equation}\label{eqcub2}
Y=\{(8-3\zeta_6)f_1+7f_2=0\}\subset\bP^4\end{equation}
with the same $G$-action given by \eqref{eqn:2a5gen}.
The aim of this section is to prove the following result.
\begin{prop}\label{prop:main4}
    Let $\mls_Y$ be a non-empty mobile $G$-invariant linear system on $Y$, and $\mu\in\dq$ such that $\mu\mls_Y\sim_\dq-K_Y$. Then the log pair $(Y,\mu\mls_Y)$ is canonical away from $\sing(Y)$.
\end{prop}
\begin{proof}
    This follows from Propositions~\ref{prop:5A2curvecenter} and~\ref{prop:points2A5Y}.
\end{proof}

 First, we classify small orbits and curves of degrees at most 11. We show that all such curves are reducible. In the second subsection, we study singularities of pairs $(Y,\mu\mls_Y)$ as above along invariant curves, and in the third subsection, we consider 0-dimensional centers.

\subsection{Small $G$-orbits and $G$-invariant curves of low degrees}
\begin{lemm}\label{lemm:2a5cubicorbit}
   A $G$-orbit of points in $Y$ with length $\leq$ 20 is one of the following:
\begin{align*}
    \Sigma_5&=\text{the orbit of}\quad[1:\zeta_6-1:-\zeta_6:\zeta_6-1:1],\\
     %\Sigma_5'&=\text{the orbit of} \quad[1:-\zeta_6:\zeta_6-1:-\zeta_6:1],\\
     \Sigma_{12}&=\text{the orbit of} \quad[\zeta_5^3:\zeta_5^2:0:\zeta_5:1],\\
    \Sigma_{12}'&=\text{the orbit of}\quad [\zeta_5^4:\zeta_5:0:\zeta_5^3:1],\\
    \Sigma_{15}&=\text{the orbit of}\quad [1:0:0:0:0],\\
     \Sigma_{20}&=\text{the orbit of}\quad 
   [(3\zeta_6 - 8):
    (-8\zeta_6 + 5):
    (5\zeta_6 + 3):
    7(\zeta_6 - 1):
    7].
\end{align*}
where the length of each orbit is indicated by the subscript. A $G$-orbit of points in $Y$ with length 30 is the orbit of a general point in one of the following two curves:
$$
\text{a cuspidal cubic curve:}\quad\{x_1-x_4=x_2-x_3=0\}\cap Y ,
$$
or
$$
\text{a line:}\quad\{x_1 + x_4=
x_2 + x_3=x_5=0\}\subset Y.
$$
Moreover, every $G$-orbit of points in $Y$ of length different from 60 is one of the orbits described above.
\end{lemm}
Observe that $\Sigma_5$ is the singular locus of $Y$. We now describe the $G$-invariant reducible curves of degrees lower than 12 on $Y$. Later, we show that they are the only $G$-invariant curves of such degrees.
\begin{lemm}\label{lemm:2A5cubreduciblecurve}
        Let $C$ be a $G$-invariant reducible curve of degree at most 11. Then $C$ is the union of curves in one of the following orbits:
    \begin{itemize}
    \item 
    one of the following two orbits of 6 lines
     \begin{multline*}
          \mathcal L_6=\text{the orbit of }\{x_1 + x_4 + (-\zeta_5^3 - \zeta_5^2)x_5=
x_2 + (\zeta_5^3 + \zeta_5^2)x_4 + \\+(\zeta_5^3 + \zeta_5^2)x_5=
x_3 -(\zeta_5^3+ \zeta_5^2)x_4 + x_5=0\},
     \end{multline*}
          \begin{multline*}
 \mathcal L_6'=\text{the orbit of }\{x_2 -(\zeta_5^3 +\zeta_5^2 + 1)x_4 - (\zeta_5^3 +\zeta_5^2 + 1)x_5=
x_1 + \\+x_4 + (\zeta_5^3 +\zeta_5^2 + 1)x_5=
x_3 + (\zeta_5^3 + \zeta_5^2 + 1)x_4 + x_5=0\},
   \end{multline*}
     \item  one of the following two orbits of 10 lines
  \begin{multline*}
 \mathcal L_{10}=\text{the orbit of }\{ x_1 + x_3 + x_5=(5\zeta_3 - 3)x_4 +7x_5=\\=7x_1 + (5\zeta_3 - 3)x_2=0\},
\end{multline*}
$$
 \mathcal L_{10}'=\text{the orbit of }\{ x_1 + x_3 + x_5=
-\zeta_3x_4 + x_5=
x_1 -\zeta_3x_2=0\}.
$$
        %\item 
       % one of the 2 orbits of 6 conics ...
    \end{itemize}
    %Note that each of the orbits above consists of pairwise disjoint components.
     Moreover, $\cL_{10}'$ consists of ten lines passing through pairs of 5 singular points of $Y$. The lines in $\cL$ are pairwise disjoint for $\cL=\cL_6, \cL_6'$ or $\cL_{10}$.
\end{lemm}
\begin{proof}
    The proof is similar to that of Lemma~\ref{lemm:12line}.
\end{proof}
The rest of this subsection is devoted to proving the following result.

\begin{prop}\label{prop:nocurve}
Let $C$ be a $G$-invariant curve in $Y$ of degree at most 11. Then $C$ is the union of all lines in one of the orbits $\cL_6$, $\cL_6'$, $\cL_{10}$, or $\cL_{10}'$ given in Lemma~\ref{lemm:2A5cubreduciblecurve}.
\end{prop}
\begin{proof}
    This follows from Lemmas \ref{lemm: 2A5degs}, \ref{lemm:no6}, \ref{lemm:no8}, and \ref{lemm:no10}.
\end{proof}

First we notice that there is a unique $G$-invariant surface in the linear system $|\mathcal O_{Y}(2)|$ and $|\mathcal O_{Y}(3)|$. We denote them by $Q$ and $R$ respectively. We have 
\begin{align}\label{eqn:orb12}
    \Sigma_{12},\Sigma_{12}'\in Q\cap R, \quad \Sigma_5\in R\setminus Q,\quad \Sigma_{20}\in Q\setminus R.
\end{align}
%e record the incidence table, where "y" means that the orbit on the top belongs to the curve on the left.
%$$
%\begin{tabular}{c|c|c|c|c|c}
%&$\Sigma_{5}$&$\Sigma_{12}$&$\Sigma_{12}'$&$\Sigma_{15}$&$\Sigma_{20}$\\\hline
  %  $Q$&y&y&y&n&y \\\hline
  %   $R$&n&y&y&y&n\\\hline
  %   $\cL_6$&n&y&n&n&n\\\hline
   %     $\cL_6'$&n&n&y&n&n \\\hline
   %    $\cL_{10}$&n&n&n&n&y\\\hline
    %   $\cL_{10}'$&y&n&n&n&n
%\end{tabular}
%$$
By computation, we find that $Q$ is a nodal K3 surface.

\begin{lemm}\label{lemm: sing X3'}
    The singular locus of $Q$ is $\Sigma_{5}$ and each singular point is an ordinary double point. The singular locus of $R$ is $\Sigma_{12}\cup\Sigma_{12}'$ and each singular point is an ordinary double point. 
\end{lemm}

\begin{lemm}\label{lemm: 5 not in C}
    Let $C\subset Y$ be an irreducible $G$-invariant curve of degree at most 11 which is not a union of ten lines. Then $\Sigma_5\not\subset C$.
\end{lemm}
\begin{proof}
    Assume that $\Sigma_5\subset C$. Let $\widetilde Y$ be the blowup of $Y$ in $\Sigma_5$, let $E$ be the exceptional divisor, and $\widetilde C$ be the strict transform of $C$. We denote by $H$ the pullback to $\widetilde Y$ of a general hyperplane section on $Y$. The base locus of the linear system $|4H-3E|$ is the strict transform of the union of the ten lines that pass through pairs of points in $\Sigma_5$. By assumption, the curve $\widetilde C$ is not in the base locus of $|4H-3E|$, so we have $(4H-3E)\cdot\widetilde C\ge0$. On the other hand, 
    \begin{align}\label{eqn:4h-3e}
    0\leq(4H-3E)\cdot\widetilde C=4d-15E_1\cdot\widetilde C\leq 44-15 E_1\cdot\widetilde C
    \end{align}
    where $E_1$ is an irreducible component of $E$.  The divisor $E_1$ is a quadric cone and is invariant under an $\aq$-action. Then $\aq$ acts faithfully on the base conic of the cone.
    %since the only normal subgroup of $\aq$ is not cyclic. 
    The orbit of a smooth point in the cone is at least of length 4. It follows that if $\widetilde C$ does not pass through the vertex $P$ of $E_1$, we have $E_1\cdot\widetilde C\ge4$, contradicting \eqref{eqn:4h-3e}. If $\widetilde C$ passes through  $P$, it is singular at $P$, because otherwise $\fA_4$ does not act faithfully on  the tangent space of $\widetilde C$ at $P$. Hence $E_1\cdot\widetilde C\ge2\cdot2=4$, again contradicting \eqref{eqn:4h-3e}.
\end{proof}

\begin{lemm}\label{lemm: 2A5degs}
    Let $C\subset Y$ be an irreducible $G$-invariant curve of degree $d\le11$. Then $d\in\{6,8,10\}$. Moreover, if $d=6$ (resp. $d=8$), then $C\subset R$ (resp. $C\subset Q$).
\end{lemm}
\begin{proof}
    By Lemma~\ref{lemm: 5 not in C}, $\Sigma_5\not\subset C$. If $C\not\subset Q$, it follows from possible lengths of orbits that
    $$
    Q\cdot C=2d=12a+20b+30c,\quad\text{ for  }a,b,c\in\bZ_{\geq0}.
    $$
    Since $d\le11$, we find that
    $$(d,a,b,c)\in\{(6,1,0,0),(10,0,1,0)\}.
    $$ 
    Similarly, if $C\not\subset R$, we have 
    $$
    R\cdot C=3d=12a+20b+30c,\quad\text{ for  }a,b,c\in\bZ_{\geq0}.
    $$ 
    Recall that $R$ is singular at the orbits of length 12 by Lemma~\ref{lemm: sing X3'}. Then $a=0$ or $a\ge2$. The only possibilities are 
    $$
    (d,a,b,c)\in\{(8,2,0,0),(10,0,0,1)\}.
    $$
    Therefore, when $d\notin\{6,8,10\}$, we know that $C\subset Q\cap R$. But $Q\cap R$ is an irreducible curve of degree $18$.
\end{proof}

Next, we show that there is also no irreducible $G$-invariant curve of degree 6, 8 or 10.

\begin{lemm}\label{lemm:no6}
    Let $C$ be a $G$-invariant curve of degree $6$ in $Y$. Then $C$ is a union of 6 lines in $\cL_6$ or $\cL_6'$ given in Lemma~\ref{lemm:2A5cubreduciblecurve}.
\end{lemm}
\begin{proof}
    Assume that $C$ is not a union of 6 lines. By Lemma \ref{lemm:2A5cubreduciblecurve}, $C$ is irreducible. Lemma \ref{lemm: 2A5degs} shows that $C\subset R$. From~\eqref{eqn:orb12}, we know that $Q$ contains both orbits of length 12. Since $C\cdot Q=12$, the curve $C$ must contain one and only one orbit of length 12 and $C$ is smooth along this orbit. Assume that $C$ contains $\Sigma_{12}$. Recall from Lemma \ref{lemm: sing X3'} that points in $\Sigma_{12}$ are nodes of $R$. Let $f\colon\widetilde{R}\rightarrow R$ be the blowup of $R$ at $\Sigma_{12}$,  $E$ its exceptional divisor, and $\widetilde C$ the strict transform of $C$. Then
    $$
    \widetilde C\sim_\dq f^*(C)-\frac{a}{2}E,
    $$
    for some positive integer $a$. By Hodge index theorem, we have $C^2\leq 4$. It follows that
    \begin{align*}
        2p_a(\widetilde C)-2&=(K_{\widetilde{R}}+\widetilde C)\cdot\widetilde C\\
        &=(f^*(H)+f^*(C)-\frac{a}{2}E)\cdot\widetilde C\\
        &=6+C^2-\frac{12a^2}{2}\\
        &\le10-6a^2.
    \end{align*}
    We deduce that $p_a(\widetilde C)\le6-3a^2\le3$. Since there is no $G$-orbits of length $\leq 3$, the geometric genus $g(\widetilde C)=p_a(\widetilde C)$, i.e., both $\widetilde C$ and $C$ are smooth and $g(C)\leq 3$. From \cite[Lemma 5.1.5]{CheltsovShramov}, we know that $C$ is a smooth rational curve. Then $C$ contains a $G$-orbit of length $20$. On the other hand, the only $G$-orbit of length 20 is contained in $Q$ by \eqref{eqn:orb12}. This contradicts $C\cdot Q=12$.
 \end{proof}

 \begin{lemm}\label{lemm:no8}
     There is no $G$-invariant curve of degree 8 in $Y$.
 \end{lemm}
 \begin{proof}
Assume that $C$ is such a curve. It is necessarily irreducible. Lemma \ref{lemm: 2A5degs} shows that $C\subset Q$. Recall from Lemma~\ref{lemm: sing X3'} that $Q$ is a K3 surface singular at $\Sigma_5$. Hence, by \cite[Proposition 6.7.3]{CS}, we have $\pic^G(Q)\cong\dz$. Let $H$ be a general hyperplane section on $Q$. Since $H^2=6$ is not a square, we know that $H$ generates $\Pic^G(Q)$. On the other hand, Lemma \ref{lemm: 5 not in C} implies that $C$ is contained in the smooth locus of $Q$, and thus is a Cartier divisor. It follows that $C=nH$, for some $n\in\bZ$. Then $8=C\cdot H=6n$, which is impossible.
 \end{proof}

\begin{lemm}\label{lemm:no10}
    Let $C$ be a $G$-invariant curve of degree $10$ in $Y$. Then $C$ is a union of 10 lines in $\cL_{10}$ or $\cL_{10}'$ given in Lemma~\ref{lemm:2A5cubreduciblecurve}.
\end{lemm}
\begin{proof}
Assume that $C$ is not a union of 10 lines. By Lemma \ref{lemm:2A5cubreduciblecurve}, $C$ is irreducible.  Consider its normalization $ f \colon C' \to C $. Since $C$ is not contained in $R\cap Q$, the proof of Lemma \ref{lemm: 2A5degs} shows that $C$ cannot have an orbit of length $12$. By \cite[Lemma 5.1.5]{CheltsovShramov}, we deduce that the genus of $C'$ satisfies
\[
g(C') \geq 6.
\]
Consider the divisor $ D = f^*(\mathcal{O}_C(3)) $ on $ C'$, and the restriction map
\[
H^0(Y, \mathcal{O}_Y(3)) \longrightarrow H^0\left(C, \mathcal{O}_C(3)\right),
\]
We want to estimate $ h^0(D)=h^0(C', \mathcal{O}_{C'}(D)) $. If $ D $ is non-special, then by Riemann–Roch:
\[
h^0(D) = \deg D - g + 1 = 30 - g + 1 \leq 25.
\]
If $ D $ is special, then by the Clifford theorem:
\[
h^0(D) \leq \frac{\deg D}{2} + 1 = 16.
\]
Consider the map $g:H^0(Y, \mathcal{O}_Y(3))\to  H^0\left(C', \mathcal{O}_{C'}(D)\right)$ given as the composition of
\[
H^0(Y, \mathcal{O}_Y(3)) \xrightarrow{\vert_C} H^0\left(C, \mathcal{O}_C(3)\right) \xrightarrow{f^*} H^0\left(C', \mathcal{O}_{C'}(D)\right).
\]
Since $ h^0(Y, \mathcal{O}_Y(3)) = 34 $, we find that the kernel of this map has dimension at least
\[
34 - h^0(D) \geq 9.
\]
This kernel consists of cubic hypersurfaces in $ Y $ that contain $ C $.

Arguing in the same way as in Lemma \ref{lemm:no8}, we see that $C$ cannot be contained in $Q$. Then  $C \cdot Q=20$, so that $C$ must contain $\Sigma_{20}$. Let $V_{15}$ and $V_{20}$ be the subspaces of $H^0(Y, \mathcal{O}_Y(3))$ consisting of cubics containing the orbit $\Sigma_{15}$ and $\Sigma_{20}$ respectively. If $\Sigma_{15}\subset C$, then $\ker\varphi\subset V_{15}\cap V_{20}.$ But we compute that $\dim(V_{15}\cap V_{20})=5<9$. Hence, the curve $C$ cannot contain $\Sigma_{15}$.

Since $ R \cdot C = 30 $, the curve $ C $ must contain an orbit of length $ 30 $ lying on $ R $. Let $\Sigma_{30}$ be such an orbit, and $V_{30}$ the space of cubics containing $\Sigma_{30}$. We have explicitly described these orbits in Lemma \ref{lemm:2a5cubicorbit}. In particular, $\Sigma_{30}$ is either the orbit of one of the following 4 points
$$
[
    -\zeta_4 - 2:
    1:
    1:
    1:
    1
],\quad 
[
    \zeta_4 - 2:
    1:
    1:
    1:
    1
],
$$
$$
[ -2\zeta_5^3 - 2\zeta_5^2:
    1:
    \zeta_5^3 + \zeta_5^2 - 1:
    \zeta_5^3 + \zeta_5^2 - 1:
    1],
    $$
    $$
    [
    2\zeta_5^3 + 2\zeta_5^2 + 2:
    1:
    -\zeta_5^3 - \zeta_5^2 - 2:
    -\zeta_5^3 - \zeta_5^2 - 2:
    1
],
$$
or the orbit of a general point in the line  
$$
\{x_1+x_4=x_2+x_3=x_5=0\}\subset Y. 
$$
A linear algebra computation then shows that $\dim(V_{20}\cap V_{30})<9$ for any such orbit $\Sigma_{30}$. Therefore we obtain a contradiction.
\end{proof}

\subsection{Invariant curves}

\begin{prop}\label{prop:5A2curvecenter}
    Let $C$ be a $G$-invariant curve in $Y$. Then each irreducible component of $C$ is not a non-canonical center of $(Y,\mu\mls_Y)$.
\end{prop}
\begin{proof}
Assume that irreducible components of $C$ are non-canonical centers of $(Y,\mu\cM_Y)$. Proposition~\ref{prop:nocurve} shows that $C$ is the union of all curves in one of the following orbits given in Lemma~\ref{lemm:2A5cubreduciblecurve}
$$
\mathcal L_6,\quad \mathcal L_6',\quad \mathcal L_{10},\quad \text{or} \quad\mathcal L_{10}'.
$$ 
    Assume that $C$ is one of the unions of six lines in $\mathcal L_6$ or $\mathcal L_6'$ which are pairwise disjoint. Let $\widetilde Y\to Y$ be the blowup of $C$ in $Y$, $E$ the exceptional divisor, and $H$ the pullback of a general hyperplane section on $Y$. One can check that $C$ is cut out by cubics, which implies that $|3H-E|$ is nef. By our assumption that irreducible components of $C$ are non-canonical centers of $(Y,\mu\mls_Y)$, we know that $|2H-mE|$ is mobile for some $m>1$, and thus $|2H-E|$ is mobile as well. It follows that the divisor $(2H-E)^2$ is effective. On the other hand, we have $(3H-E)\cdot(2H-E)^2=-6<0$, which is a contradiction. 
    
    If $C$ is the union of 10 lines in $\mathcal L_{10}$ which are pairwise disjoint, we proceed in exactly the same way as $C$ is also cut out by cubics, so $(4H-E)$ is nef and $(4H-E)\cdot(2H-E)^2=-32<0$. Finally, the case where $C$ is the union of curves in $\mathcal L_{10}'$, i.e., the union of the ten lines passing through pairs of points in $\Sigma_5$, is excluded the same way as in \cite[Proof of Proposition 3.4]{CSZ}. 
\end{proof}

\subsection{Orbits of points}
\begin{prop}\label{prop:points2A5Y}
    Let $P\in Y$ be a point and $\Sigma$ its $G$-orbit. If $P$ is a non-canonical center of $(Y,\mu\mls_Y)$, then $\Sigma=\Sigma_5$.
\end{prop}
\begin{proof}
    Assume that $P$ is a non-canonical center. By Remark~\ref{remark:Ziquan}, we know that $P$ is a non-log-canonical center of $(Y,\frac32\mu\cM_Y)$. Let $\varepsilon$ be a positive rational number such that 
    $$
    \Sigma\subset\Omega,\quad \Omega:=\mathrm{Nklt}(Y,(\frac{3}{2}-\varepsilon)\mu\cM_Y).
    $$
 Assume that $\Omega$ contains a curve $C$. Let 
$$
m=\mult_C((\frac{3}{2}-\varepsilon)\lambda\cM_X).
$$
Observe that
$$
1<m<\frac32\quad\Rightarrow \quad \frac{m^2}{m-1}>\frac92.
$$    Let $M_1, M_2\in\cM_X$ be two general elements and $H$ a general hyperplane section of $Y$.  
     Then by Theorem~\ref{theo:demailly}, we have that
    $$
    27\geq (\frac{3}{2}-\varepsilon)^2\mu^2(M_1\cdot M_2\cdot H)\geq \frac{m^2}{m-1}\deg(C)>\frac92\deg(C),
    $$
which implies that $\deg(C)<6$. Proposition~\ref{prop:nocurve} shows that such curves do not exist. It follows that $\Omega$ contains no curve.

Now notice that 
$$
K_Y+(\frac{3}{2}-\varepsilon)\mu\cM_Y+2\varepsilon\cO_Y(1)\sim_\bQ \cO_Y(1).
$$
Let $\cI$ be the multiplier ideal sheaf of $(\frac{3}{2}-\varepsilon)\mu\cM_Y$. By Nadel vanishing theorem, we have 
$$
|\Omega|\leq h^0(\cO_Y(1))=5.
$$
It follows that $\Omega=\Sigma=\Sigma_5$.
\end{proof}

\section{Proof of Theorems \ref{theo:main1}, \ref{theo:main2} and \ref{theo:main3}}\label{sect:theproof}

In this section, we explain how the results in Sections~\ref{sect:a5quad1} to~\ref{sect:a5cub2} prove Theorems~\ref{theo:main1} and~\ref{theo:main2}, and can be adapted to show Theorem~\ref{theo:main3} about the nonstandard $\fS_5$-action. For the standard $\fA_5$-action, the result readily follows from \cite{CSZ}.

\begin{proof}[Proof of Theorem~\ref{theo:main1}]
    By \cite[Section 3]{CSZ}, this follows from Proposition~\ref{prop:main1} and~\ref{prop:main2}.
\end{proof}

Similarly, in the case of the nonstandard $\fA_5$-action, Theorem~\ref{theo:main2} follows from Propositions~\ref{prop:main3} and ~\ref{prop:main4}. We explain this implication in detail.

\subsection{Nonstandard $\fA_5$-action}
We introduce some notation first. Let $X$ be the quadric given by \eqref{eqquad2}, $Y$ the cubic given by \eqref{eqcub2}, with the nonstandard $G$-action given by \eqref{eqn:2a5gen}. We denote by $\Sigma_5$ and $\Sigma_5'$ the two orbits in $X$ of length 5, $\chi$ and $\chi'$ the Cremona map associated with them respectively, and $\Sigma_5^Y$ the orbit of length 5 in $Y$. Let $\Bir^G(X)$ be the group of $G$-equivariantly birational automorphisms of $X$.

 It is well-known (see e.g., \cite[Section 3]{CSZ}, \cite[Theorem 3.3.1]{CS} and \cite{Ch-toric,CSar}) that the Noether--Fano inequalities (cf. Theorem \ref{theo:nfi}) imply that Theorem~\ref{theo:main2} follows from the~following result.

\begin{theo}\label{theorem:quadric-cubic-technical}
    Let $\mls_X$ be a non-empty mobile $G$-invariant linear system on $X$ and $\lambda\in\dq$ such that $\lambda\mls_X\sim_\dq-K_X$. Then there exists $\gamma\in\mathrm{Bir}^G(X)$ such that either $(X,\lambda\mls_X)$ or $(Y,\mu\cM_Y)$ has canonical singularities, where $\mls_{Y}$ is the proper transform of $\mls_X$ by $\chi\circ\gamma$, and $\mu\mls_{Y}\sim_\dq-K_Y$. 
\end{theo}

%In this section, we will use the results of Sections \ref{sect:a5quad1}, \ref{sect:a5cub1}, \ref{sect:a5quad2} and \ref{sect:a5cub2} to prove Theorem \ref{theorem:quadric-cubic-technical}. First, let us show that, 

First, let us explain why we need the birational automorphism $\gamma$ in the theorem above. Recall from Section~\ref{sect:a5quad2} that $(X,\lambda\cM_X)$ can have 1-dimensional non-canonical centers. Here we show that, up to replacing $\mls_X$ by its strict transform under a birational automorphism, $(X,\lambda\mls_X)$ is canonical away from the orbits of length five.
\begin{comment}
    \ref{proposition:quadric-maximal-singularities} to \ref{prop:cub21} below.

\begin{prop}
\label{proposition:quadric-maximal-singularities}
The log pair $(X,\lambda\mathcal{M}_X)$ is canonical away from $\Sigma_5\cup\Sigma_5^\prime$.
\end{prop}

\begin{prop}
\label{proposition:cubic-maximal-singularities}
The log pairs $(Y,\mu\mathcal{M}_Y)$ and $(Y,\mu^\prime\mathcal{M}_{Y}')$ are canonical away from $\mathrm{Sing}(Y)$.
\end{prop}

\begin{prop}\label{prop:quad21}
    Let $\cM_X$ be any non-empty mobile $G$-invariant linear system on the quadric $X$. Choose $\lambda\in\bQ^{>0}$ such that $\lambda\cM_X\sim_\bQ -K_X$. Then $(X,\lambda\cM_X)$ is canonical away from the union of two $G$-orbits of points of length $5$ in $X$.
\end{prop}

\begin{prop}\label{prop:cub21}
    Let $\cM_Y$ be any non-empty mobile $G$-invariant linear system on the cubic $Y$. Choose $\mu\in\bQ^{>0}$ such that $\mu\cM_Y\sim_\bQ -K_Y$. \begin{itemize}
        \item If the non-canonical center of $(Y,\mu\cM_Y)$ is $0$-dimensional, then it consists of the singular locus $\Sing(Y)$.
        \item If the non-canonical center of $(Y,\mu\cM_Y)$ is a curve, then the Sarkisov link centered at it is a $G$-equivariantly birational self-map from $Y$ to $Y$.
    \end{itemize}
\end{prop}
\end{comment}

\begin{prop}\label{prop: canonical linear system on quadrics}
    Let $\mls_X$ be a non-empty mobile $G$-invariant linear system. Then there exists $\gamma\in\mathrm{Bir}^G(X)$ such that $(X,\lambda'\mls'_X)$ is canonical away from $\Sigma_5\cup \Sigma_{5}^\prime$, where $\mls_X'=\gamma_*(\mls_X)$, and $\lambda'\in\bQ$ such that $\lambda'\mls'_X\sim_\dq-K_X$.
\end{prop}
\begin{proof}
Let $\lambda\in\bQ$ such that $\lambda\cM_X\sim_\bQ-K_X$.
    If $(X,\lambda\mls_X)$ is canonical away from $\Sigma_5\cup \Sigma_{5}^\prime$, we are done.
    Assume on the contrary that there exists a $G$-irreducible subvariety $Z$ not contained in $\Sigma_5\cup \Sigma_{5}^\prime$, and irreducible components of $Z$ are non-canonical centers of $(X,\lambda\mls_X)$. Then, by Proposition~\ref{prop:main3}, $Z$ is one of the following irreducible curves:
    \begin{itemize}
        \item rational curves $C_4$ and $C_4'$ of degree 4 given by \eqref{not:deg4},
        \item rational curves $C_8$ and $C_8'$ of degree 8 described in Remark~\ref{rema:c8curve},
        \item a smooth curve $C_{10}$ of degree 10 and genus 6.
    \end{itemize}
    Moreover, it follows from Lemma \ref{lemm:linkC8C10} and Lemma \ref{lemm:linkC4} that there exists a commutative $G$-equivariant diagram
    
    % https://q.uiver.app/#q=WzAsNCxbMCwwLCJWIl0sWzIsMCwiViciXSxbMCwxLCJYIl0sWzIsMSwiWCciXSxbMCwxLCJcXGNoaSIsMCx7InN0eWxlIjp7ImJvZHkiOnsibmFtZSI6ImRhc2hlZCJ9fX1dLFswLDIsIlxcdmFycGhpIiwyXSxbMSwzLCJcXHZhcnBoaSciXSxbMiwzLCJcXHRhdSIsMix7InN0eWxlIjp7ImJvZHkiOnsibmFtZSI6ImRhc2hlZCJ9fX1dXQ==
    \[\begin{tikzcd}
    	V && {V'} \\
    	X && {X}
    	\arrow["\chi", dashed, from=1-1, to=1-3]
    	\arrow["\varphi"', from=1-1, to=2-1]
    	\arrow["{\varphi'}", from=1-3, to=2-3]
    	\arrow["\delta"', dashed, from=2-1, to=2-3]
    \end{tikzcd}\]
    where 
    \begin{itemize}
        \item $\varphi$ is the blowup of $Z$,
        \item $\chi$ is an biregular involution if $\deg(Z)=4$, and is a composition of flops if $\deg(Z)=8$ or 10,
        \item $\varphi'$ is the blowup of a curve $Z'$ with the same degree and genus as $Z$,
        \item $\delta\in\Bir^G(X)$.
    \end{itemize}
  Set $\mls_X'=\delta_*(\mls_X)$ and $\lambda'\in\bQ$ such that  $\lambda'\mls_X'+K_X\sim_\dq0$. Since $\Pic(X)$ is generated by $\cO_X(1)$, we know that $\cM_X$ is a linear subsystem of $|\mathcal{O}_{X}(n)|$ for $n=\frac{3}{\lambda}$. Let $n'=\frac{3}{\lambda'}$. Then $\mls_X'\subset|\mathcal O_X(n')|$. We claim that $n'<n$. Indeed, let $\mls_V$ be the strict transform of the linear system $\mls_X$ on $V$.
    Note that $\codim_X(Z)=2$ and $\mult_Z(\lambda\cM_X)>1$. We have that
    $$
   0\sim_\bQ \varphi^*(K_X+\lambda\mls_X)\sim_\dq K_V+\lambda\cM_{V}+aE,\quad \text{for some }a>0.
    $$ 
    Pushing forward this class to $X$ via $\varphi\circ\chi$, we obtain that 
    $$
    K_X+\lambda\cM_X'+aD\sim_\bQ 0
    $$
    for some effective divisor $D$ on $X$.
 Since $K_X+\lambda'\cM_X'\sim_\bQ0$, it follows that $\lambda'>\lambda$, i.e., $n'<n$. 
   
   To summarize, if a curve is a non-canonical center, then we can find a $G$-equivariantly birational automorphism such that the pushforward $\cM_X'$ is a subsystem of $|\cO_X(n')|$ for $n'$ strictly  smaller than $n$. Therefore, by iterating this process, we will obtain a linear system which has no 1-dimensional non-canonical center, and thus the resulting pair is canonical away from $\Sigma_5\cup\Sigma_5'$. 
\end{proof}

We recall the following lemma from \cite{abban2024double}.
\begin{lemm}\label{lemm: ACPS}
    Let $V$ be a threefold, $K\subset\Aut(V)$ a finite subgroup fixing a smooth point $P\in V$,  $\mls_V$ a non-empty mobile $K$-invariant linear system on $V$, and $\lambda\in\dq$ such that $P$ is a non-canonical center of $(V,\lambda\mls_V)$ . If $K$ acts on the Zariski tangent space $T_P(V)$ of $V$ at $P$ via an irreducible representation, then $\mult_P(\mls_V)>\frac{2}{\lambda}$.
\end{lemm}

\begin{coro}\label{coro: mult at sigma 5} 
    Assume that $(X,\lambda\mls_X)$ is not canonical at $\Sigma_5$. Then $\mult_{\Sigma_5}\mls_X>\frac2\lambda$.
\end{coro}
\begin{proof}
    The stabilizer of a point $P\in\Sigma_5$ is isomorphic to $\aq$, which acts on $T_P(X)$ faithfully. The only 3-dimensional faithful representation of $\fA_4$ is irreducible. Then we apply the previous lemma.
\end{proof}
\begin{lemm}\label{lemm: alphaGP112}
    Let $S=\p(1,1,2)$, and $K$ a finite group acting faithfully on $S$ such that $|\mathcal O_S(1)|$ has no $K$-invariant curves. Then $\alpha_K(S)\ge\frac{1}{2}$.
\end{lemm}
\begin{proof}
    Let $L$ be a general element in $|\cO_S(1)|$. Suppose $\alpha_K(S)<\frac{1}{2}$. Then there exists a $K$-invariant effective $\dq$-divisor $D$ such that the log pair $(S,\frac{1}{2}D)$ is not log-canonical, and $D$ satisfies
    $$
 D\sim_\dq4L\sim_\dq-K_S,\quad \text{and}\quad \frac{1}{2}D=\sum a_iC_i,
    $$ 
  where for each $i$, we have $a_i\in\bZ_{\geq 0},$  and $C_i\in|\cO_S(d_i)|$ for some $d_i$. First, we show that $a_i\leq 1$ for all $i$. Indeed, we have  $2=\deg(\frac{1}{2}D)=\Sigma a_id_i$. If $a_j>1$ for some $j$, then $a_j=2$ and $d_j=1$. Since $|\cO_S(1)|$ contains no $K$-invariant curve, there exists $g\in K$ such that $a_jg(C_j)$ also shows up in $\frac12D$. This contradicts $\deg(\frac12 D)=2$.
  
  Let $\varepsilon\in\bQ_{>0}$ such that $(S,\frac{1-\varepsilon}{2}D)$ is not log-canonical. Since $a_i\leq 1$ for all $i$, we know that $\Gamma=\mathrm{Nklt}(S,\frac{1-\varepsilon}{2}D)$ does not contain any curve. Nadel vanishing theorem then implies that $\Gamma$ contains a single point, namely the vertex of $S$.
  %Using Nadel vanishing theorem gives, since $K_S+\frac{1-\varepsilon}{2}4L=-4L^+<0$, that $\Gamma$ is connected, so is a single point,namely the vertex of the cone $\Pi$. 
  Consider the blowup $\widetilde S\rightarrow S$ of the vertex and let $E$ be the exceptional divisor. We have that
  $$
  \widetilde S\cong\f_2,\quad K_{\widetilde S}\sim f^*(K_S),\quad \widetilde D=f^*(D)-mE,\quad \widetilde L=f^*(L)-\frac{1}{2}E
  $$
  for some $m\in\bZ_{>0}$. It follows that
    \begin{align*}
        0&\le\widetilde D\cdot\widetilde L=2-m,\quad \text{and thus}\quad m\frac{1-\varepsilon}{2}<1.
    \end{align*}
Since
 $$
 f^*(K_S+\frac{1-\varepsilon}{2} D)\sim_\bQ K_{\widetilde S}+\frac{1-\varepsilon}{2}\widetilde D+m\frac{1-\varepsilon}{2}E,
 $$
we see that $(\widetilde S,\frac{1-\varepsilon}{2}\widetilde D+m\frac{1-\varepsilon}{2}E)$ is not log-canonical at some point in $E$. 
Since $m\frac{1-\varepsilon}{2}<1$, the pair $(\widetilde S,\frac{1-\varepsilon}{2}\widetilde D+E)$ is also not log-canonical at some point in $E$. 
By inversion of adjunction, we see that $(E,\frac{1-\varepsilon}{2}\widetilde D\vert_E)$ is not log-canonical. Note that $\widetilde D\vert_E\sim_\dq-E\vert_E$ is a divisor of degree $2m$ on $E=\pl$. It follows that $\alpha_K(\pl)<1$. Recall from \cite{CS} that
    \[
    \alpha_K(\mathbb{P}^1) = 
    \begin{cases}
    \frac{1}{2} & \text{if } K \cong C_n, \\
    1 & \text{if } K \cong \fD_n, \\
    2 & \text{if } K \cong \fA_4, \\
    3 & \text{if } K \cong \fS_4, \\
    6 & \text{if } K \cong \fA_5.
    \end{cases}
    \]
    By our assumption that $|\cO_S(1)|$ has no $K$-invariant curves, we see that $K$ is not a cyclic group. Therefore, we obtain a contradiction and this completes the proof. 
\end{proof}

\begin{coro}\label{Coro: discr at sigma 5 for usual blowup}
Assume that points of $\Sigma_5^Y$ are centers of non-canonical singularities of $(Y,\mu\mls_Y)$. Consider $\pi\colon\widetilde Y\rightarrow Y$, the blowup  of $Y$ in $\Sigma_5^Y$. Let $m\in\dq$ such that $\pi^*(\mu\mls_{Y})\sim_\dq\mu\mls_{\widetilde Y}+mE$, where $\mls_{\widetilde Y}$ is the strict transform of $\cM_Y$ to $\widetilde Y$, and $E$ is the exceptional divisor of $\pi$. Then $m>1$.
\end{coro}
\begin{proof}
Let $P$ be a point of $\Sigma_5$, and let $F$ be the component of $E$ that is mapped to $P$. 
Then $F\simeq\pp(1,1,2)$,  since  $P$ is an $\sA_2$-singularity. 
Observe that 
$$
\pi^*(K_Y+\mu\mls_{Y})\sim_\dq K_{\widetilde Y}+\mu\mls_{\widetilde Y}+(m-1)E.
$$
Recall that $P$ is a non-canonical center of  $(Y,\mu\mls_Y)$. It follows that $(\widetilde Y,\mu\mls_{\widetilde Y}+(m-1)E)$ is not canonical at some point in $F$.
Hence, $(\widetilde Y,\mu\mls_{\widetilde Y}+mE)$ is not log-canonical at some point in $F$.

Now assume that $m\leq 1$. 
Then $(\widetilde Y,\mu\mls_{\widetilde Y}+E)$ is not log-canonical at some point in $F$. It follows from the inversion of adjunction that $(F,\mu\mls_{\widetilde Y}\vert_F)$ is not log-canonical. Note that 
 \begin{align}\label{eqn:F2m}
     \mu\mls_{\widetilde Y}\vert_F\sim_\dq-mF\vert_F\sim_\dq\cO_F(2m). 
 \end{align}
The stabilizer of $P$ is isomorphic to $\fA_4$, which acts faithfully on $F$.
 Then \eqref{eqn:F2m} implies that $\alpha_{\fA_4}(F)<\frac{m}{2}$, which contradicts Lemma \ref{lemm: alphaGP112}. Therefore we conclude that $m>1$.
\end{proof}

We are now ready to prove Theorem \ref{theorem:quadric-cubic-technical}.
\begin{proof}[Proof of Theorem \ref{theorem:quadric-cubic-technical}]

    By Proposition \ref{prop: canonical linear system on quadrics}, we may assume that, up to applying a $G$-equivariantly birational automorphism of $X$, the log pair $(X,\lambda\cM_X)$ is canonical away from $\Sigma_5\cup\Sigma_5'$. Since $\Sigma_5$ and $\Sigma_5'$ are exchanged by some element in the normalizer of $G$ in $\Aut(X)$, we can further assume that $(X,\lambda\cM_X)$ is canonical away from $\Sigma_5$. Now, it suffices show that either $(X,\lambda\cM_X)$ is canonical along $\Sigma_5$, or $(Y,\mu\cM_Y)$ is canonical along $\Sigma_5^Y$.

     %Let $m=\mult_{\Sigma_5}\mls_X$, and $m'=\mult_{\Sigma_5}\mls_Y$, where $m'$ is defined as in Corollary~\ref{Coro: discr at sigma 5 for usual blowup}.
     
     We denote by $H_X$ (resp. $H_Y$) a general hyperplane section on $X$ (resp. $Y$).  Let $n,n'\in\dz$ such that $\mls_X\sim_\dq nH_X$, and $\mls_Y\sim_\dq n'H_Y$. Note that $n=\frac{3}{\lambda}$ and $n'=\frac2\mu$. Recall from \cite[Section 3]{CSZ} that the Cremona map $\chi$ fits into the $G$-equivariant commutative diagram:
     $$
     \xymatrix{
&V\ar@{->}[ld]_{g}\ar@{-->}[rr]^{\rho}&&W\ar@{->}[rd]^{f}&\\
X\ar@{-->}[rrrr]^{\chi}&&&&Y
}
$$
 where $g$ is the blowup of $\Sigma_5$, $\rho$ is a small birational map that flops the proper transforms of 10 conics that contain three points in $\Sigma_5$, and $f$ contracts to $\Sigma_5^Y$ the proper transforms of 5 hyperplane sections of $X$ that pass through four points in $\Sigma_5'$. 
   %  $$
%\xymatrix{
%&V\ar@{->}[ld]_{g}\ar@{->}[rd]^{f}&\\
%X\ar@{-->}[rr]^{\chi}&&Y
%}
%$$
 %where $g$ is the blowup of $\Sigma_5$ and $f$ is the blowup of $\Sigma_5^Y$. {\color{red}{right???}}
 Let $\widetilde H_X$, $\widetilde H_Y$, $\widetilde \cM_X$, $\widetilde \cM_Y$ be the strict transforms in $V$ of $H_X$, $H_Y$, $\cM_X$, $\cM_Y$ respectively, $E$ the exceptional divisor of $g$, and $F$ the strict transform in $V$ of the exceptional divisor of $f$. We compute
    $$
    \widetilde\mls_X=n\widetilde H_X-mE=(4n-5m)\widetilde H_Y-(3n-4m)F=n'\widetilde H_Y-m'F=\widetilde\mls_Y,
    $$ 
 which yields 
    $$
    n'=4n-5m,\quad m'=3n-4m,
    $$
for some $m,m'\in\bQ$.   By Corollary \ref{coro: mult at sigma 5}, if $(X,\lambda\mls_X)$ is not canonical along $\Sigma_5$, then $\lambda m>2$, i.e., $3m>2n$. On the other hand, if $(Y,\mu\mls_Y)$ is not canonical at $\Sigma_5^Y$, then by Corollary \ref{Coro: discr at sigma 5 for usual blowup}, we have $1<\mu m'$, i.e., $2n>3m$. These two cases cannot happen simultaneously. We conclude that either $(X,\lambda\mls_X)$ is canonical at $\Sigma_5$, or $(Y,\mu\mls_Y)$ is canonical at $\Sigma_5^Y$.
\end{proof}

\subsection{Nonstandard $\fS_5$-action}
Throughout this subsection, $G'=\fS_5$. Using our analysis of the non-standard $\fA_5$-action, it is not hard to prove Theorem~\ref{theo:main3}. Let $X$ be the same quadric threefold defined by \eqref{eqquad2} and $Y$ the cubic threefold defined by \eqref{eqcub2}. We consider the {\em nonstandard} $G'$-action on $X$ and $Y$ generated by the nonstandard $\fA_5$-action \eqref{eqn:2a5gen}, and an extra involution 
$$
\iota:(\mathbf x)\mapsto (x_3,x_4,x_1,x_2,-x_1-x_2-x_3-x_4-x_5).
$$
We will show that
    the only $G'$-Mori fibre spaces that are $G'$-equivariantly birational to $X$ are $X$ and $Y$.

Recall from  Propositions~\ref{prop:curves2A5X} and ~\ref{prop:ptcase2} that the possible non-canonical centers under the nonstandard $\fA_5$-action are points in $\Sigma_5$ and $\Sigma_5'$, curves $C_4, C_4'$, $C_8, C_8'$, or some curves of degree 10. 

Under the $G'$-action, the orbits $\Sigma_5$ and $\Sigma_5'$ are still invariant. So, the cubic $Y$ with $5\sA_2$-singularities is still $\fS_5$-equivariantly birational to $X$. Proposition~\ref{prop:main4} clearly also holds for the $\fS_5$-action. We focus on the quadric. 
 Any curve of degree 10 becomes irrelevant, since if it is not $\fS_5$-invariant, its $\fS_5$-orbit becomes a curve of degree 20, which exceeds the bound 18 as in Remark~\ref{rema:degreebound1}.

The involution $\iota\in \fS_5$ exchanges the curve $C_4$ with $C_4'$, and $C_8$ with $C_8'$. We show that $C_8$ and $C_8'$ are not non-canonical centers in this case.
\begin{lemm}\label{lemm:C8+C8'iscanonical}
    The curves $C_8$ and $C_8'$ are not centers of non-canonical singularities of $(X,\lambda\mls_X)$.
\end{lemm}
\begin{proof}
  Assume $C_8$ is a non-canonical center. Since $\cM_X$ is $G'$-invariant, $C_8'$ is also a non-canonical center. Put $Z=C_8+C_8'$. Then we have $\mult_Z(\lambda\mls_X)>1$. Let $\pi:\widetilde X\to X$ be the blowup of $X$ along $Z$, $E$ the exceptional divisor, and $H$ the pullback of a general hyperplane section on $X$ to $\widetilde X$. Similarly as before, the assumption that $C_8$ and $C_8'$ are non-canonical centers implies that $(3H-E)^2$ is effective. Using equations, we check that the linear system $|\cO_X(5)-C_8-C_8'|$ does not have base curves other than $C_8+C_8'$. It follows that $(5H-E)$ is nef. Let us compute the intersection number $(3H-E)^2\cdot(5H-E)$. We have
        \begin{itemize}
            \item $H^3=2,$
            \item $H^2\cdot E=0,$
            \item $H\cdot E^2=-\deg(C_8+C_8')=-16,$
            \item $E^3=-\deg(\mathcal N_{C_8/X})-\deg(\mathcal N_{C_8'/X})=K_X\cdot(C_8+C_8')+4=-44.$
        \end{itemize}
  Then
    \begin{align*}
        (3H-E)^2\cdot(5H-E)&=45H^3 - 39H^2E + 11HE^2 - E^3\\
        &=90-0-176+44\\
        &=-42.
    \end{align*}
    This contradicts the fact that $(5H-E)$ is nef.
\end{proof}

%\begin{lemm}\label{lemm:s5curvesinR+R'}
    %Let $C$ be a $G$-invariant curve contained in $R\cup R'$, with the notation of \eqref{not:RR'}. Then $C$ is not a center of a non-canonical singularities of $(X,\lambda\mls_X)$.
%\end{lemm}
%\begin{proof}
  %  It is similar to the proof of Proposition \ref{lemm:quad2curveinR}.
%\end{proof}

On the other hand, $C_4+C_4'$ can indeed be a non-canonical center. We present the Sarkisov link centered at $C_4+C_4'$.
One can check that the linear system $|\cO_X(3)-(C_4+C_4')|$ gives rise to a rational map $\tau:X\dashrightarrow\bP^3$ fitting into a $G'$-equivariant commutative diagram
$$
\xymatrix{
&\widetilde X\ar@{->}[ld]_{\pi}\ar@{->}[rd]^{\rho}&\\
X\ar@{-->}[rr]^{\tau}&&\bP^3
}
$$
where $\pi$ is the blowup along $C_4+C_4'$ and the resolution $\rho$ of $\tau$ is a double cover ramified along a singular sextic surface. The involution of this double cover gives rise to a $G'$-equivariantly birational involution
$$
\delta: X\dashrightarrow X.
$$
Note that $\delta$ naturally commutes with $G'$ in $\Bir^{G'}(X)$. It follows that $\delta\notin\Aut(X)=\mathsf{PSO}_5(\bC)$ since no element in $\Aut(X)$ centralizes $G'$.
Similarly as in Lemma~\ref{lemm:linkC8C10}, we can show that $-K_{\widetilde X}$ is big and nef, and that $|n(-K_{\widetilde X})|$ gives a small birational map for $n\gg0$. Namely, $\delta$ also fits into a Sarkisov link as in \eqref{diag}. 

\begin{proof}[Proof of Theorem~\ref{theo:main3}]
   Note that the Sarkisov link centered at $C_4+C_4'$ is again a $G'$-equivariantly birational selfmap. Thus, Proposition~\ref{prop: canonical linear system on quadrics} and Theorem~\ref{theorem:quadric-cubic-technical} still hold for the $G'$-action. The argument in the previous subsection applies and Theorem~\ref{theo:main3} is proved in the same way. 
\end{proof}

\end{document}